\documentclass[12pt]{amsart}
\usepackage{SSdefn}
\usepackage{eucal}
\usepackage{tikz-cd}
\usepackage{float}
\usepackage{braids}
\usepackage{caption}

\newcommand{\FI}{\mathbf{FI}}
\newcommand{\FS}{\mathbf{FS}}
\newcommand{\gen}{\mathrm{gen}}
\newcommand{\uconf}{\mathrm{uConf}}

\DeclareMathOperator{\maxdeg}{maxdeg}
\DeclareMathOperator{\cor}{cor}

\newcommand{\coloneq}{\mathrel{\mathop:}\mkern-1.2mu=}

\title[{Periodicity in the cohomology of symmetric groups via divided powers}]{Periodicity in the cohomology of symmetric groups \\via divided powers}
\date{\today}

\author{Rohit Nagpal}
\address{Department of Mathematics, University of Chicago, Chicago, IL}
\email{\href{mailto:nagpal@math.uchicago.edu}{nagpal@math.uchicago.edu}}
\urladdr{\url{http://math.uchicago.edu/~nagpal/}}

\author{Andrew Snowden}
\address{Department of Mathematics, University of Michigan, Ann Arbor, MI}
\email{\href{mailto:asnowden@umich.edu}{asnowden@umich.edu}}
\urladdr{\url{http://www-personal.umich.edu/~asnowden/}}

\thanks{AS was supported by NSF grants DMS-1303082 and DMS-1453893 and a Sloan Fellowship .}

\subjclass[2010]{%
	20C30, %   	Representations of finite symmetric groups
	20J06%   	Cohomology of groups
	}

\begin{document}

\begin{abstract}
A famous theorem of Nakaoka asserts that the cohomology of the symmetric group stabilizes. The first author generalized this theorem to non-trivial coefficient systems, in the form of $\FI$-modules over a field, though one now obtains periodicity of the cohomology instead of stability. In this paper, we further refine these results. Our main theorem states that if $M$ is a finitely generated $\FI$-module over a noetherian ring $\bk$ then $\bigoplus_{n \ge 0} \rH^t(S_n, M_n)$ admits the structure of a $\bD$-module, where $\bD$ is the divided power algebra over $\bk$ in a single variable, and moreover, this $\bD$-module is ``nearly'' finitely presented. This immediately recovers the periodicity result when $\bk$ is a field, but also shows, for example, how the torsion varies with $n$ when $\bk=\bZ$. Using the theory of connections on $\bD$-modules, we establish sharp bounds on the period in the case where $\bk$ is a field. We apply our theory to obtain results on the modular cohomology of Specht modules and the integral cohomology of unordered configuration spaces of manifolds. 
\end{abstract}

\maketitle
\tableofcontents

\section{Introduction}

\subsection{The main theorem} 
A famous theorem of Nakaoka asserts that the cohomology groups of the symmetric groups stabilize: for $n>2t$, the restriction map
\begin{displaymath}
\rH^t(S_n, \bk) \to \rH^t(S_{n-1}, \bk)
\end{displaymath}
is an isomorphism, for any coefficient ring (or even abelian group) $\bk$, with the symmetric groups acting trivially. It is natural to ask if there is some way to generalize this theorem to allow non-trivial coefficients. Thus suppose that for each $n \ge 0$ we have a representation $M_n$ of $S_n$ over a ring $\bk$. We can then consider the groups $\rH^t(S_n, M_n)$, for $t$ fixed and $n$ varying. Obviously, to hope for any sort of relationship between these groups we must assume that the $M_n$ form a ``coherent system'' of representations, in some sense.

In recent years, several kinds of algebraic structures have been studied that can rightfully be called ``coherent systems'' of representations. In this paper, we focus on $\FI$-modules, popularized by Church, Ellenberg, and Farb \cite{fimodules}, though see \S \ref{ss:future} for additional discussion. An $\FI$-module over $\bk$ can be defined as a system of representations $\{M_n\}_{n \ge 0}$, as above, with transition maps $M_n \to M_{n+1}$ satisfying certain conditions; see \S \ref{ss:fi} below. In a general $\FI$-module, the transition maps could all be~0, and so the $M_n$'s need not be related. However, in a \emph{finitely generated} $\FI$-module, the $M_n$'s are very closely related. In this setting, the first author has found a generalization of Nakaoka's theorem:

\begin{theorem}[\cite{nagpal}] \label{nagpalthm}
Suppose $\bk$ is a field of characteristic $p$. Let $M$ be a finitely generated $\FI$-module over $\bk$, and fix $t \ge 0$. Then $\dim_{\bk}{\rH^t(S_n, M_n)}$ is periodic for $n$ sufficiently large, with period a power of $p$.
\end{theorem}

This is a fine theorem, but is somewhat deficient in that it is numerical rather than structural; that is, instead of simply having an equality of dimensions, one would like to have some sort of additional algebraic structure  inducing isomorphisms between appropriate cohomology groups. The main purpose of this paper is to establish this structure.

Let $M$ be a finitely generated $\FI$-module. Define
\begin{displaymath}
\Gamma^t(M) = \bigoplus_{n \ge 0} \rH^t(S_n, M_n),
\end{displaymath}
regarded as a graded $\bk$-module in the evident manner. We show that $\Gamma^t(M)$ canonically has the structure of a module over the divided power algebra $\bD$ over $\bk$ in a single variable (see \S \ref{s:dp} for the definition). The ring $\bD$ is not noetherian in general, but it is coherent whenever $\bk$ is noetherian \cite[Theorem~4.1]{dp}, and so finitely presented $\bD$-modules are reasonably well-behaved. Our main result is the following theorem.

\begin{theorem}[Main theorem] \label{mainthm}
Assume that $\bk$ is a commutative noetherian ring. Let $M$ be a finitely generated $\FI$-module over $\bk$ and let $t \ge 0$. Then there exists a finitely presented graded $\bD$-module $K$, depending functorially on $M$, and a map of $\bD$-modules $\Gamma^t(M) \to K$ that is an isomorphism in all sufficiently large degrees.
\end{theorem}

It is easy to see that a finitely presented $\bD$-module over a field of positive characteristic has eventually periodic dimensions (see \S \ref{sec:dp-positive}), and so the above theorem recovers Theorem~\ref{nagpalthm}. However, it is stronger than Theorem~\ref{nagpalthm} in two ways. For one, the $\bD$-module structure provides more flexibility. For example, suppose that $M \to N$ is a map of finitely generated $\FI$-modules over a field of positive characteristic. It follows from Theorem~\ref{mainthm} that the image of the map $\Gamma^t(M) \to \Gamma^t(N)$ has eventually periodic dimension; Theorem~\ref{nagpalthm} does not give this information. Theorem~\ref{mainthm} also improves on Theorem~\ref{nagpalthm} in that it does not require $\bk$ to be a field. Thus, for example, Theorem~\ref{mainthm} yields non-trivial information about how the torsion in $\rH^t(S_n, M_n)$ varies with $n$ in the case $\bk=\bZ$.

%Another major advantage of our approach is that it yields a cleaner proof than the one in \cite{nagpal}. The proof in \cite{nagpal} is long and relies on intricate calculations with cocycles. Our proof, by contrast, is conceptual and straightforward.

\begin{remark}
If $M$ is a finitely generated $\FI$-module then $\Gamma^t(M)$ need not be a finitely presented (or even finitely generated) $\bD$-module, even for $t=0$; see Example~\ref{ex:nfp}. Thus Theorem~\ref{mainthm} is in some sense optimal.
\end{remark}

\begin{remark}
Theorem~\ref{nagpalthm} also follows from a recent result of Harman (\cite[Proposition~3.3]{nate}). However, the arguments there are specific to fields.
\end{remark}

\begin{remark}
There are similar, but easier, results for the homology of $\FI$-modules: for a finitely generated $\FI$-module $M$, one can give $\bigoplus_n \rH_t(S_n, M_n)$ the structure of a finitely generated $\bk[t]$-module. It follows that $\dim_{\bk} \rH_t(S_n, M_n)$ stabilizes as $n \to \infty$.
\end{remark}

\subsection{Quantitative results}

Let $M$ be a finitely generated $\FI$-module over a field $\bk$ of characteristic $p$. According to Theorem~\ref{nagpalthm}, $\Gamma^t(M)$ has eventually periodic dimensions. This raises two questions: what is the eventual period, and when does periodicity set in? We prove the following theorem that addresses these questions:

\begin{theorem} \label{thm:quant}
Suppose that $M$ is generated in degrees $\le g$ with relations in degrees $\le r$ and has degree $\delta$. Let $q$ be the smallest power of $p$ such that $\delta<q$. Then
\begin{displaymath}
\dim_{\bk}{\rH^t(S_n, M_n)}=\dim_{\bk}{\rH^t(S_{n+q}, M_{n+q})}
\end{displaymath}
holds for all $n \ge \max(g+r, 2t+\delta)$.
\end{theorem}

We recall that the degree $\delta$ of $M$ is the degree of the eventual polynomial $n\mapsto \dim_{\bk}(M_n)$, and is bounded above by $g$. The above theorem greatly improves the bounds from \cite{nagpal}, and we believe it is close to optimal.

\subsection{Application: configuration spaces}

Let $\uconf_n(\cM)$ be the unordered configuration space of $n$ points on a manifold $\cM$. In \S \ref{sec:config}, we prove, under some mild assumptions on $\cM$, that there is a finitely presented $\bD$-module $K$ and $\bk$-module isomorphisms
\begin{displaymath}
\rH^t(\uconf_n(\cM), \bk)_n \cong K_n
\end{displaymath}
for $n \gg 0$. When $\bk$ is a field of positive characteristic, this recovers (without bounds) recent periodicity results on these cohomology groups (see \cite{jeremy} for the sharpest result and summary). When $\bk=\bZ$, this provides new information about how the torsion in the cohomology of $\uconf_n(\cM)$ varies with $n$. See Example~\ref{example:sphere} for the case of $\cM=\cS^2$, where the torsion has long been well-understood.

\subsection{Application: cohomology of Specht modules}

Let $\bk$ be a field of characteristic $p$. For a partition $\mu$, let $\bM_{\mu}$ be the associated Specht module over $\bk$. Write $\mu[n]$ for the partition $(n - |\mu|, \mu)$ (defined only for $n \ge \mu_1 + |\mu|$). In \S \ref{sec:specht}, we prove the following:

\begin{theorem}
Let $\mu$ be a partition and let $q$ be the smallest power of $p$ greater than $\vert \mu \vert$. Then
\begin{displaymath}
\dim_{\bk} \rH^t(S_n, \bM_{\mu[n]}) = \dim_{\bk} \rH^t(S_n, \bM_{\mu[n+q]})
\end{displaymath}
holds for $n \ge \max(2t + d, 2 d + \mu_1)$.
\end{theorem}

Note that a similar result can be obtained from  \cite[Corollary~2.6]{nate} (but without bounds on the onset of periodicity) and that this is a specific case of a more general open problem \cite[Problem~8.3.1]{hemmer}.  

\subsection{Overview of the proofs}

The main innovation of this paper is the systematic use of the $\bD$-module structure on $\Gamma^t(M)$. This not only leads to stronger results, but proofs that are much cleaner than those in \cite{nagpal}.

We deduce Theorem~\ref{mainthm} from a generalization of a theorem of Dold and a structurual result on $\FI$-modules due to the first author. The argument is roughly as follows. Suppose that $V$ is a $\bk[S_d]$-module, and let $\cI(V)$ be the associated induced $\FI$-module (see \S \ref{ss:induced}). We prove that $\Gamma^t(\cI(V))$ is a finitely generated and relatively free $\bD$-module (that is, it has the form $M_0 \otimes_{\bk} \bD$ for some finitely generated $\bk$-module $M_0$.); this is our generalization of Dold's theorem. Now suppose that $M$ is an arbitary finitely generated $\FI$-module. The first author has shown that there is a sequence of $\FI$-modules
\begin{displaymath}
0 \to M \to I^0 \to \cdots \to I^n \to 0
\end{displaymath}
where each $I^k$ is semi-induced (i.e., has a filtration with induced quotients) and whose homology is torsion. This gives a spectral sequence that computes $\Gamma^t(M)$ from the $\Gamma^i(I^j)$'s, ignoring finitely many degrees (due to the fact that the homology is possibly non-zero). Our generalization of Dold's theorem shows that each $\Gamma^i(I^j)$ is finitely generated and free. Thus, using the fact that $\bD$ is coherent, it follows that $\Gamma^t(M)$ is isomorphic to a finitely presented $\bD$-module outside of finitely many degrees.

To prove our generalization of Dold's theorem, and also to prove the quantitative bounds in Theorem~\ref{thm:quant}, we appeal to the theory of connections. The ring $\bD$ admits a derivation $d$ defined by $d(x^{[n]})=x^{[n-1]}$. A connection on a $\bD$-module $M$ is a $\bk$-linear map $\nabla \colon M \to M$ satisfying a version of the Leibniz rule with respect to $d$. A fundamental result (Proposition~\ref{prop:connection}) asserts that if $M$ admits a connection then $M$ is relatively free. Suppose that $V$ is the trivial representation of $S_0$, so that $\cI(V)_n=\bk$ for all $n$. We show that the restriction map
\begin{displaymath}
\rH^t(S_n, \bk) \to \rH^t(S_{n-1}, \bk)
\end{displaymath}
induces a connection on the $\bD$-module $\Gamma^t(\cI(V))$. A similar construction applies to arbitrary representations $V$ of $S_d$. This is how we prove our generalization of Dold's theorem. Similar ideas appear in Dold's original paper, but the use of connections greatly clarifies the arguments.

Suppose now that $I \to J$ is a map of induced $\FI$-modules over a field $\bk$ of characteristic $p$. The induced map $\Gamma^t(I) \to \Gamma^t(J)$ need not respect the connections on these $\bD$-modules. However, we show that if $I$ and $J$ are generated in degrees $<q$, where $q$ is a power of $p$, then this map does respect the $q$-fold iterate of the connections. From this, we deduce that if $M$ is a finitely generated $\FI$-module of degree $<q$ then $\Gamma^t(M)$ admits a connection with respect to the $q$-fold iterate of $d$. This is the key step in the proof of Theorem~\ref{thm:quant}.

\subsection{Open problems}
\label{ss:future}

As mentioned, $\FI$-modules provide one way to produce ``coherent systems'' of $S_n$-representations, but there are others. We mention two specific ones.

For a positive integer $d$, let $\FI_d$ be the category whose objects are finite sets, and where morphisms are injections together with a $d$-coloring of the complement of the image. Let $M$ be a finitely generated $\FI_d$-module over a field $\bk$. Then $M_n$ is a representation of $S_n$. The first author has conjectured that $\dim_{\bk} \rH^t(S_n, M_n)$ is eventually a quasi-polynomial in $n$ of degree at most $d-1$; note that this exactly reduces to Theorem~\ref{nagpalthm} when $d=1$. One can give these groups the structure of a module over the divided power algebra in $d$ variables; ideally, this module structure would somehow be responsible for the quasi-polynomiality.

Let $\FS$ be the category whose objects are finite sets, and where morphisms are surjections. Let $\FS^{\op}$ be the opposite category, and let $M$ be a finitely generated $\FS^{\op}$-module over $\bk$. Once again, $M_n$ is a representation of $S_n$. Given $n \le m$ there is a natural map $M_m \to M_n$ of $S_n$ representations, and thus a restriction map
\begin{displaymath}
\rH^t(S_m, M_m) \to \rH^t(S_n, M_n).
\end{displaymath}
When $M$ is the trivial $\FS^{\op}$-module, these maps are isomorphisms for $n>2t$: this is exactly Nakaoka's theorem. Are these maps always isomorphisms for $n \gg 0$? One can also ask an analogous question for the homology groups  $\rH_t(S_n, M_n)$ in this setting. 

In \cite{rww}, Randal-Williams and Wahl have found a general condition which, given a monoidal category $(\cC, \oplus, 0)$ and a functor $F \colon \cC \to \Mod_{\bZ}$, ensures  that the homology groups \[\rH_t(\Aut_{\cC}(A \oplus X^{\oplus n}), F(A \oplus X^{\oplus n})) \] stabilize in $n$. Does their condition imply that the corresponding cohomology groups are eventually periodic? We thank the referee for suggesting this question.

\subsection{Outline of paper}

In \S \ref{sec:background} we review the theory of $\FI$-modules. The material here is mostly not new. In \S \ref{s:dp} we introduce the ring $\bD$, and prove a number of basic facts about $\bD$-modules. The main results of this paper are proved in \S \ref{sec:proof-main}. Finally, \S \ref{sec:specht} and \S \ref{sec:config} contain applications of our results to the cohomology of Specht modules and configuration spaces respectively.

%We start with providing background on $\FI$-modules and refining some structural results on them in \S \ref{sec:background}. In particular, we study the saturation functor \[\bS \colon \Mod_{\FI} \to \Mod_{\FI}\] and its derived functor $\rR\bS$. Our main results of this section are Theorem~\ref{thm:triangle-induced} and Corollary~\ref{cor:derived-saturation} which state that if $M$ is a finitely generated $\FI$-module then $\rR\bS (M)$ is quasi isomorphic to a bounded complex of induced modules. We have devoted \S \ref{s:dp} to a study of modules over $\bD$. $\bD$-modules admitting a connection are studied in \S \ref{sec:connections}, and the $\epsilon$ and $\lambda$ invariants of $\bD$-modules in positive characteristic are studied in \S \ref{sec:dp-positive}. 
%
%Cohomology of $\FI$-modules is dealt in \S \ref{sec:FI-cohomology}. We prove in \S \ref{sec:induced-connection} that if $M$ is an induced $\FI$-module then $\Gamma^t(M)$ is a $\bD$-module admitting a connection. This is used in \S \ref{sec:proof-main} to prove Theorem~\ref{mainthm}.

%Finally, \S \ref{sec:specht} and \S \ref{sec:config} contain applications of our results to the cohomology of Specht modules and configuration spaces respectively.  

\subsection*{Acknowledgement} We would like to thank Jeremy Miller for pointing out that the orientability assumption is not needed in Theorem~\ref{thm:configuration-spaces}. We thank the anonymous referee for a careful reading and for pointing out a mistake in an earlier version of Proposition~\ref{prop:tensor-functor}.

\section{Background on $\FI$-modules}
\label{sec:background}

\subsection{$\FI$-modules} \label{ss:fi}

Let $\FI$ be the category whose objects are finite sets and whose morphisms are injections. An {\bf $\FI$-module} with coefficients in $\bk$, a commutative noetherian ring, is a functor from $\FI$ to the category $\Mod_{\bk}$ of $\bk$-modules. We write $\Mod_{\FI}$ for the category of $\FI$-modules. Suppose that $M$ is an $\FI$-module. We write $M_n$ for the value of $M$ on the finite set $[n]=\{1,\ldots,n\}$. This is a representation of $S_n$ over the ring $\bk$. The inclusion $[n] \to [n+1]$ induces an $S_n$-equivariant map $M_n \to M_{n+1}$. In fact, one can think of an $\FI$-module as a sequence $\{M_n\}_{n \ge 0}$ of representations equipped with transition maps $M_n \to M_{n+1}$ satisfying the following condition: the image of $M_n$ in $M_{n+r}$ lands in $M_{n+r}^{S_r}$ for all $n$ and $r$, see \cite[Remark~3.3.1]{fimodules}. 

\subsection{The twisted commutative algebra $\bA$}

A {\bf representation of $S_{\ast}$} over $\bk$ is a sequence $(V_n)_{n \ge 0}$ where $V_n$ is a $\bk[S_n]$-module. We write $\Rep_{\bk}(S_{\ast})$ for the category of such sequences. If $V$ and $W$ are two representations of $S_{\ast}$, we define their tensor product $V \otimes W$ to be the representation of $S_{\ast}$ given by
\begin{displaymath}
(V \otimes W)_n = \bigoplus_{i+j=n} \Ind_{S_i \times S_j}^{S_n}(V_i \otimes_{\bk} W_j).
\end{displaymath}
One can show that this tensor product is naturally symmetric, that is, there is a canonical isomorphism $V \otimes W \to W \otimes V$ that squares to the identity.

A {\bf twisted commutative algebra} is a commutative algebra object in the tensor category $\Rep(S_{\ast})$, that is, an $S_{\ast}$-representation $A$ equipped with a multiplication map $A \otimes A \to A$ satisfying the appropriate axioms. (We note that, by Frobenius reciprocity, specifying such a multiplication map is the same as specifying $S_i \times S_j$ equivariant maps $A_i \otimes_{\bk} A_j \to A_{i+j}$ for all $i$ and $j$.) Given such an algebra $A$, an $A$-module is an $S_{\ast}$-representation $M$ equipped with a multiplication map $A \otimes M \to M$ satisfying the appropriate axioms.

Let $\bA$ be the $S_{\ast}$-representation given by $\bA_n=\bk$ with the trivial $S_n$ action for all $n$. There is a canonical map $\bA_i \otimes_{\bk} \bA_j \to \bA_{i+j}$ for every $i$ and $j$, and this defines a multiplication $\bA \otimes \bA \to \bA$ that gives $\bA$ the structure of a twisted commutative algebra. Suppose $M$ is an $\bA$-module. Then $M$ is an $S_{\ast}$-representation, and the multiplication map $\bA_1 \otimes_{\bk} M_n \to M_{n+1}$ gives an $S_n$-equivariant map $M_n \to M_{n+1}$. These maps give $\{M_n\}_{n \ge 0}$ the structure of an $\FI$-module. In this way, the categories $\Mod_{\bA}$ and $\Mod_{\FI}$ are equivalent, see \cite[Proposition~7.2.5]{catgb}. We freely pass between the two perspectives.

\subsection{Generators and relations}

Since $\bA$ is a commutative algebra object in a symmetric tensor category, there is a notion of tensor product of $\bA$-modules and (since there are enough projectives) one can derive this. Regard $\bk$ as an $\bA$-module by letting the elements of positive degree act by~0. The $\bA$-modules $\Tor^{\bA}_i(M, \bk)$ are important quantities associated to the $\FI$-module $M$. We let $g(M)$ be the maximal degree occurring in this module for $i=0$, and $r(M)$ the maximal degree for $i=1$. One can interpret $g(M)$ and $r(M)$ as the degrees of generators and relations for $M$. In terms of $\FI$-modules, $\Tor^{\bA}_0(M, \bk)$ assigns to $[n]$ the quotient of $M_n$ by the $S_n$-subrepresentation generated by the image of the transition map $M_{n-1} \to M_n$. The derived functors of this construction, i.e., what we call $\Tor^{\bA}_{\bullet}(-, \bk)$, is sometimes called ``$\FI$-module homology'' in the literature (see the introduction of \cite{castelnuovo-regularity}, for example).

\subsection{Induced and semi-induced $\FI$-modules}
\label{ss:induced}

Let $V$ be a representation of $S_d$. Define $\cI(V)$ to be the $S_{\ast}$-representation given by
\begin{displaymath}
\cI(V)_n = \Ind_{S_d \times S_{n-d}}^{S_n}(V \boxtimes \mathrm{triv}),
\end{displaymath}
where here $\mathrm{triv}$ denotes the trivial representation of $S_{n-d}$. One easily verifies that $\cI(V)$ has the structure of an $\FI$-module; see \cite[Definition~2.2.2]{fimodules} for details (note that there the notation $M(V)$ is used in place of $\cI(V)$). We extend the construction of $\cI(V)$ additively to objects $V$ of $\Rep_{\bk}(S_{\ast})$. We call $\FI$-modules of the form $\cI(V)$ {\bf induced $\FI$-modules}. From the tca perspective, $\cI(V)$ is the $\bA$-module $\bA \otimes V$. From this point of view, the following proposition is immediate:

\begin{proposition}
\label{prop:tensor-hom-adjunction}
Let $V \in \Rep_{\bk}(S_{\ast})$. Then for any $\FI$-module $M$ we have
\begin{displaymath}
\Hom_{\Mod_{\FI}}(\cI(V), M) = \Hom_{S_{\ast}}(V, M).
\end{displaymath}
\end{proposition}

We say that an $\FI$-module $M$ is {\bf semi-induced} if it has a finite length filtration where the graded pieces are induced. We have the following useful result:

\begin{proposition}[{\cite[Proposition~A.8, Theorem~A.9]{djament}}]
\label{prop:short-exact-sequence-semi-induced}
In a short exact sequence of $\FI$-modules if any two of the objects are semi-induced then so is the third. 
\end{proposition}

\begin{remark}
In characteristic~0, induced $\FI$-modules are projective, and so every semi-induced $\FI$-module is induced. This is not true in general.
\end{remark}

\subsection{Finiteness properties}

An $\FI$-module is {\bf finitely generated} if it is a quotient of $\cI(V)$ for some finitely generated object $V$ of $\Rep_{\bk}(S_{\ast})$ (meaning each $V_n$ is finitely generated as a $\bk$-module and all but finitely many vanish). This definition  coincides with the general categorical notion of a finitely generated object. A fundamental result on $\FI$-modules is the following noetherianity theorem, first proved in \cite[Theorem~A]{fi-noeth} and later reproved in \cite[Corollary~7.2.7]{catgb}:

\begin{theorem}
If $M$ is a finitely generated $\FI$-module then $M$ is noetherian, that is, any $\FI$-submodule of $M$ is also finitely generated.
\end{theorem}

\subsection{The shift functor}

Let $V$ be an $S_{\ast}$ representation. We define a new $S_{\ast}$ representation $\Sigma(V)$ by $\Sigma(V)_n=V_{n+1} \vert_{S_n}$. We call this the {\bf shift} of $V$. If $M$ is an $\FI$-module then $\Sigma(M)$ is obtained by precomposing $M$ by the endofunctor $- \sqcup \{\star\}$ of $\FI$, and hence $\Sigma(M)$ is an $\FI$-module. Moreover, in this case, there is a canonical map of $\FI$-modules $M \to \Sigma(M)$. If $V$ is an $S_{\ast}$-representation then one has the important identity $\Sigma(\cI(V))=\cI(V) \oplus \cI(\Sigma(V))$. We make use of the projection map $\Sigma^n(\cI(V)) \to \cI(V)$ throughout without further mention of this identity.

\subsection{Torsion $\FI$-modules}

Let $M$ be an $\FI$-module. We say that an element $x \in M_n$ is {\bf torsion} if $x$ maps to~0 under some transition map $M_n \to M_m$. We say that $M$ itself is {\bf torsion} if all of its elements are.  We write $\Mod_{\FI}^{\tors}$ for the category of torsion $\FI$-modules. This is a Serre subcategory of $\Mod_{\FI}$ closed under direct sum, and so is a localizing subcategory.

\subsection{The degree of an $\FI$-module}

We inductively define classes $\cC_d$ of $\FI$-modules as follows. First, $\cC_{-1}$ is the class of torsion $\FI$-modules. Having defined $\cC_d$, we let $\cC_{d+1}$ be the class of $\FI$-modules $M$ such that the cokernel of the natural map $M \to \Sigma(M)$ belongs to $\cC_d$. We now define the {\bf degree} of an $\FI$-module $M$, denoted $\delta(M)$, to be the minimal $d$ such that $M \in \cC_d$, or $\infty$ if not such $d$ exists.  We have the following basic facts:
\begin{enumerate}
\item $\delta(M)=-1$ if and only if $M$ is torsion.
\item If $M$ is semi-induced then $\delta(M) = g(M)$.
\item If $N$ is a subquotient of $M$ then $\delta(N) \le \delta(M)$. In particular, $\delta(M) \le g(M)$.
\item If $M$ is an extension of $M'$ by $M''$ then $\delta(M)=\max(\delta(M'), \delta(M''))$.
\item If $M$ is non-torsion then $\delta(\coker(M \to \Sigma^k(M)))=\delta(M)-1$, for any nonzero $k$.
\item Suppose $\bk$ is a field and $M$ is finitely generated. There then exists a polynomial $p$ such that $p(n)=\dim_{\bk}{M_n}$ for all $n \gg 0$ \cite[Theorem~B]{fi-noeth}. We have $\delta(M)=\deg(p)$, using the convention that the degree of the zero polynomial is $-1$.
\end{enumerate}
The above properties are all reasonably straightforward exercises. Some details can be found in \cite{djament-vespa}, where $\FI$-modules of degree~$n$ are treated as degree~$n$ weakly polynomial functors on the category $\Theta=\FI$.

\begin{remark} We note that $\delta(M)  = g(\Sigma^n M)$ for $n$ large enough (see Theorem~\ref{thm:shift}). In the case when $\bk$ is a field we can take property (f) above as the definition of $\delta(M)$ and then the rest of the properties easily follow. This is the case that we use.
\end{remark}

\subsection{Resolutions by induced and semi-induced modules}

The following is a fundamental result on the structure of $\FI$-modules due to the first author \cite[Theorem~A]{nagpal}.

\begin{theorem} \label{thm:shift}
Let $M$ be a finitely generated $\FI$-module. Then $\Sigma^n(M)$ is semi-induced for $n \gg 0$.
\end{theorem}

As a consequence of this theorem, one obtains the following result \cite[Theorem~A]{nagpal}, which is of paramount importance to this paper:

\begin{theorem} \label{thm:semires}
Let $M$ be a finitely generated $\FI$-module of degree $d$. Then there exists a complex
\begin{displaymath}
0 \to M \to I^0 \to \cdots \to I^d \to 0
\end{displaymath}
with torsion cohomology, where $I^k$ is finitely generated and semi-induced of degree $\le \delta(M)-k$.
\end{theorem}

It is easy to deduce Theorem~\ref{thm:semires} from Theorem~\ref{thm:shift}, as follows. Given $M$, let $n \gg 0$ such that $\Sigma^n(M)$ is semi-induced. The map $M \to \Sigma^n(M)$ has torsion kernel and $\delta(\coker) \le \delta(M) -1$. Thus, by induction on degree, there exists a complex as in the theorem for the cokernel, and this gives one for $M$ as well, with $I^0=\Sigma^n(M)$. The proof of the theorem applies equally well if $M$ is a complex, rather than a module, and leads to the following result:

\begin{theorem} \label{thm:triangle}
Let $M^{\bullet}$ be a finite length complex of finitely generated $\FI$-modules. Then, in the derived category, there exists an exact triangle
\begin{displaymath}
T^{\bullet} \to M^{\bullet} \to I^{\bullet} \to
\end{displaymath}
where $T^{\bullet}$ is a finite length complex of finitely generated torsion modules and $I^{\bullet}$ is a finite length complex of finitely generated semi-induced modules.
\end{theorem}

The following result is a refinement of \cite[Theorem~A.9(3)]{djament}.

\begin{proposition}
\label{prop:induced-resolution}
Let $M$ be a semi-induced module with $\delta(M) \le d$. Then $M$ admits a resolution $F_{\bullet} \to M \to 0$ of length at most $d + 1$ such that each $F_i$ is induced.
\end{proposition}

\begin{proof}
Since $g(M)  \le d$, $\Tor^{\bA}_0(\bk, M)_r  = 0$ for $r > d$. Let $r$ be the least integer such that $\Tor^{\bA}_0(\bk, M)_r$ is non-trivial (if no such $r$ exists then $M=0$ by the Nakayama lemma and there is nothing to prove). We prove by downward induction on $r$ that there is a resolution $F_{\bullet}$ of $M$ of length $d - r+1$ where each $F_i$ is a direct sum of induced modules.  Let $F_0 = \bigoplus_{0 \le k \le d} \cI(V_k)$ where $V_k = M_k$. We note that    $\Tor^{\bA}_0(\bk,M)_r = V_r = \Tor^{\bA}_0(\bk,F_0)_r$ and $\Tor^{\bA}_0(\bk,M)_k = 0= \Tor^{\bA}_0(\bk,F_0)_k $ for $k<r$. By construction, $\delta(F_0) \le d$ and there is a surjection $\psi \colon F_0 \to M$ (Proposition~\ref{prop:tensor-hom-adjunction}). We have $\Tor^{\bA}_0(\bk,\ker(\psi))_k =0$ for $k \le r$ (semi-induced modules have no higher Tor; \cite[Theorem~B]{ramos}). By Proposition~\ref{prop:short-exact-sequence-semi-induced}, $\ker(\psi)$ is semi-induced. Clearly, we have $\delta(\ker(\psi)) \le d$. Thus by downward induction,   $\ker(\psi)$ admits a resolution of the desired format. We can append $F_0$ to this resolution to get a resolution of $M$, completing the proof.
\end{proof}

In the proof above, the construction of $F_0$ (and thus $F_i$) is functorial in $M$, that is, if $f \colon M \to N$ is a map of semi-induced $\FI$-modules with $\delta(M), \delta(N) \le d$ then there is a natural map $F_0(f) \colon F_0(M) \to F_0(N)$ making the obvious diagram commute. Thus the proof of the proposition applies equally well if $M$ is a complex of semi-induced modules generated in degree $\le d$, rather than a module. Combining this observation with Theorem~\ref{thm:triangle}, we obtain:

\begin{theorem} \label{thm:triangle-induced}
Let $M^{\bullet}$ be a finite length complex of finitely generated $\FI$-modules with $\delta(M^i) \le d$ for each $i$. Then, in the derived category, there exists an exact triangle
\begin{displaymath}
T^{\bullet} \to M^{\bullet} \to I^{\bullet} \to
\end{displaymath}
where $T^{\bullet}$ is a finite length complex of finitely generated torsion modules and $I^{\bullet}$ is a finite length complex consisting of direct sums of induced modules with $\delta(I^i) \le d$ for each $i$.
\end{theorem}

\subsection{Derived saturation}

For our purposes, it will be useful to have a better understanding of the complex in Theorem~\ref{thm:semires}. To state the relevant results, it will be convenient to introduce  derived saturation. In characteristic~0, this theory was fully developed in \cite{symc1}. We quickly sketch here how the ideas work over general rings $\bk$.

Define the {\bf generic category}, denoted $\Mod_{\FI}^{\gen}$, to be the Serre quotient of $\Mod_{\FI}$ by the torsion subcategory $\Mod_{\FI}^{\tors}$. We let $\rT \colon \Mod_{\FI} \to \Mod_{\FI}^{\gen}$ be the localization functor, which is exact, and write $\rS \colon \Mod_{\FI}^{\gen} \to \Mod_{\FI}$ for the right adjoint of $\rT$, the section functor. (This adjoint exists by the general theory of Grothendieck abelian categories:  $\Mod_{\FI}^{\tors}$ is a localizing subcategory.) We let $\bS=\rS \circ \rT$, the saturation functor. The saturation functor is left-exact, and we write $\rR \bS$ for its right-derived functor. We say that an $\FI$-module $M$ is {\bf saturated} (resp. {\bf derived saturated}) if the natural map $M \to \bS(M)$ (resp.\ $M \to \rR \bS(M)$) is an isomorphism. By \cite[Corollary~A.4]{djament}, $M$ is saturated if and only if $M$ is torsion free and the cokernel of the natural map $M \to \Sigma^k(M)$ is torsion free for each $k$.

%\Acom{Rohit:I don't think that we need local cohomology. We can just keep everything in terms of saturation and its derived functors. }
%
%For an $\FI$-module $M$, we let $\Gamma_{\fm}(M)$ be the maximal torsion submodule of $M$. The functor $\Gamma_{\fm}$ is left-exact, and we write $\rR \Gamma_{\fm}$ for its right-derived functor. We also put $\rH^i_{\fm} = \rR^i \Gamma_{\fm}$, and refer to this as local cohomology. We write $h^i(M)$ for the maximum degree occurring in $\rH^i_{\fm}(M)$.
%
%If $M$ is a bounded-below complex of $\FI$-modules then there is an exact triangle \Acom{Rohit: Do we have a reference for this? I thought this is something that we were going to prove in ficharp. Do we really need this?}
%\begin{displaymath}
%\rR \Gamma_{\fm}(M) \to M \to \rR \bS(M) \to
%\end{displaymath}
%In particular, if $M$ is an $\FI$-module then there is a short exact sequence
%\begin{displaymath}
%0 \to \rH^0_{\fm}(M) \to M \to \bS(M) \to \rH^1_{\fm}(M) \to 0
%\end{displaymath}
%and isomorphisms $\rH^i_{\fm}(M)=\rR^{i-1} \bS(M)$ for $i \ge 2$.

\begin{theorem}[{\cite[Theorem~A.9]{djament}}]
\label{thm:perfect}
A finitely generated $\FI$-module is semi-induced if and only if it is derived saturated.
\end{theorem}

\begin{corollary}
\label{cor:derived-saturation}
Let $M$ be a finitely generated $\FI$-module, and let $M \to I^{\bullet}$ be the complex of Theorem~\ref{thm:semires}. Then $\rR \bS(M)$ is quasi-isomorphic to $I^{\bullet}$. In particular,  $\rR^i \bS(M)$ is finite for all $i \ge 0$ and vanishes for $i \gg 0$.
\end{corollary}

\begin{proof}
Clearly, $\rT(I^{\bullet})$ is a resolution of $\rT(M)$ (if $N$ is a torsion module then  $\rT(N) = 0$). By Theorem~\ref{thm:perfect}, $\rT(I^k)$ is an $\rS$-acyclic for each $k$. Thus $\rR \bS(M) = \rS(\rT(I^{\bullet})) =\rR \bS(I^{\bullet})$.  The remaining assertions now follow immediately from the properties of the complex $I^{\bullet}$.
\end{proof}

%\begin{corollary}
%Let $M$ be a finitely generated $\FI$-module, and let $M \to I^{\bullet}$ be the complex of Theorem~\ref{thm:semires}. Then $\rR \bS(M)$ is quasi-isomorphic to $I^{\bullet}$.
%\end{corollary}
%
%\begin{proof}
%Apply $\rR \bS$ to the triangle in Theorem~\ref{thm:triangle}, and use the obvious fact that $\rR \bS(T)=0$ if $T$ is torsion and the less obvious fact that $\rR \bS(I)=I$ if $I$ is semi-induced. \Acom{Rohit: we have a detailed proof of this is ficharp. May be we should have a list of stuff that we need and put it in ficharp instead.}
%\end{proof}

The above corollary shows that the complex $I^{\bullet}$ in Theorem~\ref{thm:semires} is well-defined up to quasi-isomorphism and depends functorially on $M$ in the derived category.  For a non-negatively graded $\bk$-module $M$, we define {\bf max-degree}, denoted $\maxdeg M$, to be the smallest $n \ge -1$ such that $M_k = 0$ for all $k >n$. We denote $\maxdeg \rR^i\bS(M)_n$  by $s^i(M)$.

%\begin{theorem}
%Let $M$ be a finitely generated $\FI$-module. Then:
%\begin{enumerate}
%\item $h^0(M) \le g(M)+r(M)-1$.
%\item $h^1(M) \le 2g(M)-2$.
%\item $h^i(M) \le 2\delta(M)-i$ for $i \ge 2$.
%\end{enumerate}
%\end{theorem}
%
%\begin{proof}
%\Acom{deduce from \cite[Theorem~1.3]{li}}.
%\end{proof}

\begin{theorem}
\label{thm:homology-bound}
 We have $s^i(M) \le 2 \delta(M) -2 i$ for $i \ge 1$ and, $\delta(\bS(M)) = \delta(M)$. Moreover, the complex $M \to I^{\bullet}$ in Theorem~\ref{thm:semires} is exact in $\FI$ degrees $\ge g(M)+\max(g(M), r(M))$.
\end{theorem}
\begin{proof} 
By Corollary~\ref{cor:derived-saturation}, $\rR \bS(M)$ is quasi-isomorphic to $I^{\bullet}$. Since $\rH^i(I^{\bullet})$ is torsion for $i \ge 1$ (Theorem~\ref{thm:semires}), we see that $s^i(M) = \maxdeg \rH^i(I^{\bullet})$ is bounded by the max-degree of the torsion submodule of $\coker(I^{i-1} \to I^i)$. By \cite[Theorem~F]{castelnuovo-regularity}, we conclude that $s^i(M) \le 2 \delta(M) - 2 i$ for $i \ge 1$. Since the kernel of $M \to \bS(M)$ is torsion, we see that $\delta(\bS(M)) \ge \delta(M)$. On the other hand, we have $\bS(M) = \ker(I^0 \to I^1)$ (Corollary~\ref{cor:derived-saturation}), which implies that $\delta(\bS(M)) \le \delta(I^0) \le \delta(M)$. Thus $\delta(\bS(M)) = \delta(M)$.

For the last assertion, it suffices to bound the max-degrees of cohomologies of the complex $M \to I^{\bullet}$ at $M$ and $I^0$ (the rest have already been taken care of by the previous paragraph). By \cite[Theorem~F]{castelnuovo-regularity}, these max-degrees are bounded by $ g(M)+r(M) - 1$ and $\delta(M) + g(M) -1$ respectively. The inequality $\delta(M) \le g(M)$ implies that all cohomology groups of the complex $M \to I^{\bullet}$ are supported in ($\FI$) degrees $\le g(M)+\max(g(M), r(M)) -1$. This completes the proof.
\end{proof}

\section{The divided power algebra and its modules}
\label{s:dp}

\subsection{Generalities}

Fix a commutative noetherian ring $\bk$. Let $\bD$ be the divided power algebra  over $\bk$ in a single variable $x$. Thus $\bD$ has a $\bk$-basis consisting of elements $x^{[n]}$ with $n \in \bN$ in which multiplication is given by
\begin{displaymath}
x^{[n]} x^{[m]} = \tbinom{n+m}{n} x^{[n+m]}.
\end{displaymath}
We regard $\bD$ as graded, with $x^{[n]}$ having degree $n$. All $\bD$-modules we consider are graded. The ring $\bD$ is typically not noetherian: for example, if $\bk$ has characteristic $p$ then there is an isomorphism of $\bk$-algebras
\begin{displaymath}
\bk[y_0, y_1, y_2, \ldots]/(y_0^p, y_1^p, y_2^p, \ldots) \stackrel{\sim}{\to} \bD
\end{displaymath}
defined by $y_i \mapsto x^{[p^i]}$. Nonetheless, $\bD$ is always coherent \cite[Theorem~4.1]{dp}, and so finitely presented $\bD$-modules form an abelian category.

For a $\bD$-module $M$, put $\tau_{\ge n}(M)=\bigoplus_{i \ge n} M_i$; this is a $\bD$-submodule of $M$. Following \cite[\S 9]{dp}, we say that a $\bD$-module $M$ is {\bf nearly finitely presented} if there exists a finitely presented $\bD$-module $N$, called a {\bf weak fp-envelope} of $M$, and an isomorphism of $\bD$-modules $\tau_{\ge n}(M) \cong \tau_{\ge n}(N)$. We remark that if $\bk$ is $p$-adically complete (e.g., if $\bk$ has characteristic $p$) then weak fp-envelopes are unique up to canonical isomorphism \cite[Proposition~9.5]{dp}. As an example, we note that the ideal $\bD_{+} = (y_0, y_1, \ldots)$ is not finitely generated but it is nearly finitely presented as we have $\tau_{\ge 1} (\bD_{+}) \cong \tau_{\ge 1}(\bD)$.
%In this case, we define $\nu(M)$ to be the minimal integer so that the canonical map $M \to N$ is an isomorphism in all degrees $\ge \nu(M)$.

\begin{remark}
One of the first systematic appearances of divided power algebras appeared in \cite[Chapter III]{roby}. They were later famously used in Grothendieck's theory of crystalline cohomology. A modern treatment can be found in \cite[Tag~09PD]{stacks}. To the best of our knowledge, however, the results in this section do not occur in the literature.
\end{remark}

\subsection{Connections}
\label{sec:connections}

Let $d \colon \bD \to \bD[1]$ be the $\bk$-linear map defined by $x^{[k]} \mapsto x^{[k-1]}$, where we use the convention that $x^{[n]}=0$ if $n<0$. We leave to the reader the simple verification that $d$ is a derivation. A {\bf connection} on a $\bD$-module $M$ is a $\bk$-linear map $\nabla \colon M \to M[1]$ that satisfies the Leibnitz rule $\nabla(fm)=f \nabla(m) + d(f) m$ for all $f \in \bD$ and $m \in M$. The following is a fundamental result about connections:

\begin{proposition}
\label{prop:connection}
Let $M$ be a $\bD$-module equipped a connection $\nabla$. Put $\overline{M}=M/\bD_+ M$, and write $m_0$ for the image of $m \in M$ in $\overline{M}$. Define a map\begin{displaymath}
\Phi \colon M \to \bD \otimes_{\bk} \overline{M}, \qquad
\Phi(m) = \sum_{n \ge 0} x^{[n]} (\nabla^n m)_0.
\end{displaymath}
Then:
\begin{enumerate}
\item $\Phi$ is an isomorphism of $\bD$-modules.
\item Under $\Phi$, the connection $\nabla$ corresponds to $d \otimes 1$.
\item The map $\Phi$ identifies $\ker(\nabla)$ with $\overline{M}$.
\item The natural map $\ker(\nabla) \otimes_{\bk} \bD \to M$ is an isomorphism of $\bD$-modules.
\end{enumerate}
\end{proposition}

\begin{proof}
(a) Let $m \in M$ and $f \in \bD$. We have the identity
\begin{displaymath}
\nabla^n(fm) = \sum_{i+j=n} \binom{n}{i} d^i(f) \nabla^j(m).
\end{displaymath}
Thus
\begin{displaymath}
\begin{split}
\Phi(fm)
&= \sum_{n \ge 0} \sum_{i+j=n} \binom{n}{i} x^{[n]} d^i(f)_0 \nabla^j(m)_0 \\
&= \left( \sum_{i \ge 0} x^{[i]} d^i(f)_0 \right) \left( \sum_{j \ge 0} x^{[j]} \nabla^j(m)_0 \right) \\
&= f \Phi(m).
\end{split}
\end{displaymath}
In the last step we used Taylor's theorem: for any $f \in \bD$ we have
\begin{displaymath}
f = \sum_{i \ge 0} x^{[i]} d^i(f)_0.
\end{displaymath}
We have thus show that $\Phi$ is $\bD$-linear. It is clear that $\Phi_0$ is the identity map $\overline{M} \to \overline{M}$, and so $\Phi$ is an isomorphism by the following lemma (basically Nakayama's).

(b) We have
\begin{displaymath}
(d \otimes 1) \Phi(m) = \sum_{n \ge 1} x^{[n-1]} (\nabla^n m)_0
= \sum_{n \ge 0} x^{[n]} (\nabla^{n+1} m)_0 = \Phi(\nabla m).
\end{displaymath}

(c) By part~(b), $\Phi$ induces an isomorphism $\ker(\nabla) \to \ker(d \otimes 1) = 1 \otimes \overline{M}$.

(d) The map in question is just the inverse to $\Phi$, under the identification $\ker(\nabla)=\overline{M}$.
\end{proof}

\begin{lemma}
Let $f \colon M \to N$ be a map of $\bD$-modules.
\begin{enumerate}
\item Suppose the map $\overline{f} \colon \overline{M} \to \overline{N}$ is surjective. Then $f$ is surjective.
\item Suppose $\overline{f}$ is an isomorphism and $N$ has the form $\bD \otimes_{\bk} N'$ for some $\bk$-module $N'$. Then $f$ is an isomorphism.
\end{enumerate}
\end{lemma}

\begin{proof}
(a) We proceed by induction on degree. Thus let $n \in N$ have degree $d$, and suppose surjectivity holds in smaller degree. We can write $n_0=f_0(m_0)$ for some $m \in M$, and thus $n-f(m) \in \bD_+ N$, and so $n-f(m)=\sum_{k \ge 1} x^{[k]} n_k$ for some $n_k \in N$. Since then $n_k$ have strictly smaller degree than $n$, we have $n_k=f(m_k)$ for some $m_k \in M$, by induction. Thus $n-f(m)=f(m')$, where $m'=\sum_{k \ge 1} x^{[k]} m_k$, and so $n=f(m+m')$.

(b) Consider the exact sequence
\begin{displaymath}
0 \to K \to M \to N \to 0
\end{displaymath}
Reducing mod $\bD_+$, and using the fact that $\Tor^{\bD}_1(N, \bk)=0$, gives
\begin{displaymath}
0 \to \overline{K} \to \overline{M} \to \overline{N} \to 0.
\end{displaymath}
Since $f_0$ is an isomorphism, we find $\overline{K}=0$, and so $K=0$ (use part~(a) with the map $0 \to K$). This proves the lemma.
\end{proof}

Suppose now that $\bk$ has characteristic $p$ and let $q$ be a power of $p$. Let $d_q \colon \bD \to \bD[q]$ be the $q$-fold iterate of $d$, i.e., $d_q(x^{[k]})=x^{[k-q]}$. Then $d_q$ is a derivation, since any $q$-fold iteration of a derivation in characteristic $p$ is still a derivation. A {\bf $q$-connection} on a $\bD$-module $M$ is a $\bk$-linear map $\nabla \colon M \to M[q]$ satisfying the Leibnitz rule with respect to $d_q$. For example, if $\nabla \colon M \to M[1]$ is a connection then the $q$-fold iterate of $\nabla$ is a $q$-connection. Let $\bD^{(q)} = \bigoplus_{q \mid n} \bD_n$, a subalgebra of $\bD$. There is an algebra isomorphism $\bD^{(q)} \to \bD$ mapping $x^{[qn]}$ to $x^{[n]}$, under which $d_q$ corresponds to $d$. Thus Proposition~\ref{prop:connection} yields:

\begin{proposition}
\label{prop:q-connection}
Let $M$ be a $\bD$-module admitting a $q$-connection. Then there is an isomorphism of $\bD^{(q)}$-modules $M \cong \bD^{(q)} \otimes_{\bk} M'$ for some graded $\bk$-module $M'$.
\end{proposition}

\subsection{Periodicity phenomena}
\label{sec:dp-positive}

We assume in this section that $\bk$ is a field of characteristic $p$. It is not difficult to see (and we will prove) that if $M$ is a finitely presented $\bD$-module then the dimension of $M_n$ is eventually periodic in $n$ with period a power of $p$. The main purpose of this section is to introduce invariants $\epsilon$ and $\lambda$ that control the period and the onset of periodicity, and prove a few facts about them.

We first introduce some notation. Put $y_i=x^{[p^i]}$, and recall that $\bD$ is isomorphic to $\bk[y_i]_{i \ge 0}/(y_i^p)$. For a subset $I$ of $\bN$, let $\bD_I$ be the subalgebra of $\bD$ generated by the $y_i$ with $i \in I$. We also put $\bD_{<r}=\bD_{[0,r)}$ and $\bD_{\ge r}=\bD_{[r,\infty)}$, where we use the usual interval notation for subsets of $\bN$. We note here that $\bD_{I \sqcup J} \cong  \bD_{I } \otimes_{\bk} \bD_{ J}$.

\begin{proposition}
\label{d:free}
Let $f \colon M_1 \to M_2$ be a map of finitely generated free $\bD_{\ge r}$-modules. Suppose that, in a suitable basis, the matrix entries of $f$ only involve the variables $y_r, \ldots, y_{s-1}$, for some $s \ge r$. Then the kernel, cokernel, and image of $f$ are free as $\bD_{\ge s}$-modules.
\end{proposition}

\begin{proof}
Let $\ol{M}_i$ be a free $\bD_{[r,s)}$-module with the same basis as $M_i$, and define $\ol{f} \colon \ol{M}_1 \to \ol{M}_2$ using the same matrix that defines $f$. Then $f$ is obtained from $\ol{f}$ by applying the exact functor $-\otimes_{\bk} \bD_{\ge s}$. Thus $\ker(f)=\ker(\ol{f}) \otimes_{\bk} \bD_{\ge s}$ is free over $\bD_{\ge s}$, and similarly for the cokernel and image.
\end{proof}

\begin{corollary}
Let $M$ be a finitely presented $\bD$-module. Then $M$ is free as a $\bD_{\ge r}$-module for some $r$.
\end{corollary}

\begin{proof}
Apply the proposition to a presentation of $M$.
\end{proof}

\begin{definition}
Let $M$ be a finitely presented $\bD$-module. We define $\epsilon(M)$ to be the minimal non-negative integer $r$ so that $M$ is free as a $\bD_{\ge r}$-module.
\end{definition}

\begin{remark}
If $s \ge r$ then $\bD_{\ge r}$ is free as a $\bD_{\ge s}$-module. It follows that if $M$ is a finitely presented $\bD$-module then $M$ is free over $\bD_{\ge r}$ for any $r \ge \epsilon(M)$.
\end{remark}

Let $M$ be a finitely presented $\bD$-module. For $r \ge \epsilon(M)$, let $g_r(M)$ be the maximal degree of a basis element of $M$ as a $\bD_{\ge r}$-module (this is independent of the choice of basis). Let $s \ge r$. Then $\bD_{\ge r}$ is free as a $\bD_{\ge s}$-module, with basis consisting of the monomials in variables $y_r, \ldots, y_{s-1}$. The monomial $m=y_r^{p-1} \cdots y_{s-1}^{p-1}$ is the one of maximal degree, and has degree $p^s-p^r$. Thus if $e$ is a maximal degree basis element of $M$ as a $\bD_{\ge r}$-module then $m e$ is a maximal degree basis element of $M$ as a $\bD_{\ge s}$-module. It follows that we have the identity
\begin{displaymath}
g_s(M) = p^s-p^r+g_r(M).
\end{displaymath}
This observation is the basis of the following definition.

\begin{definition}
Let $M$ be a finitely presented $\bD$-module. We define $\lambda(M)$ to be the common value of the expression $g_r(M)+1-p^r$, for $r \ge \epsilon(M)$.
\end{definition}

\begin{example}
We have $\epsilon(\bD[d])=0$ and $\lambda(\bD[d])=d$.
\end{example}

The importance of the $\lambda$ and $\epsilon$ invariants comes from the following proposition.

\begin{proposition}
\label{prop:stability-period}
Let $M$ be a finitely presented $\bD$-module. Then $\dim_{\bk}(M_n)$ is periodic in $n$, for $n$ sufficiently large, with period a power of $p$. More precisely, we have $\dim(M_n)=\dim(M_m)$ whenever $n \equiv m \pmod{p^{\epsilon(M)}}$ and $n,m \ge \lambda(M)$.
\end{proposition}

\begin{proof}
Put $r=\epsilon(M)$ and $q= p^r$ and $a=\lambda(M)$. Write $M=V \otimes \bD_{\ge r}$ where $V$ is a finite dimensional graded vector space concentrated in degree $<q+a$. The degree $n$ piece of $\bD_{\ge r}$ is one dimensional if $n$ is a non-negative multiple of $q$, and 0 otherwise. We thus see that 
\begin{displaymath}
\dim(M_n) = \sum_{\substack{0 \le k \le n \\ k \equiv n \pmod{q}}} \dim(V_k).
\end{displaymath}
Thus if $n \equiv m \pmod{q}$ and $n \le m$ then
\begin{displaymath}
\dim(M_m)-\dim(M_n) = \sum_{\substack{n<k\le m \\ k \equiv n \pmod{q}}} \dim(V_k).
\end{displaymath}
Thus if $n \ge a$ then $k \ge q+a$ for every $k$ appearing in the sum, and so $V_k=0$ for such $k$. We thus find $\dim(M_m)=\dim(M_n)$.
\end{proof}

\begin{remark}
The value $\lambda(M)$ for the onset of periodicity is optimal. However, $p^{\epsilon(M)}$ is not, in general, the minimal period. For example, let $M=\bD_{\ge 1}[0] \oplus \cdots \oplus \bD_{\ge 1}[p-1]$. Then $M_n$ is one-dimensional for all $n$, and so $n \mapsto \dim(M_n)$ has period~1. But $\epsilon(M)=1$.
\end{remark}

We now explain how one can bound $\epsilon$ and $\lambda$ in certain situations.

\begin{proposition}
\label{d:ext}
Let
\begin{displaymath}
0 \to M_1 \to M_2 \to M_3 \to 0
\end{displaymath}
be a short exact sequence of finitely presented $\bD$-modules. Then
\begin{displaymath}
\epsilon(M_2) \le \max(\epsilon(M_1), \epsilon(M_3)), \qquad
\lambda(M_2) \le \max(\lambda(M_1), \lambda(M_3)).
\end{displaymath}
\end{proposition}

\begin{proof}
If $M_1$ and $M_3$ are free over $\bD_{\ge r}$ then so is $M_2$; furthermore, a basis for $M_2$ can be obtained from bases of $M_1$ and $M_3$, and so $g_r(M_2)$ is equal to $\max(g_r(M_1), g_r(M_3))$. The proposition follows easily from these observations.
\end{proof}

\begin{proposition}
\label{d:lambda}
Let $M$ be a finitely presented $\bD$-module and let $N$ be a finitely presented subquotient of $M$. Then $\lambda(N) \le \lambda(M)$.
\end{proposition}

\begin{proof}
Let $I_r$ be the maximal ideal of $\bD_{\ge r}$, i.e., $(y_r, y_{r+1}, \ldots)$. Note that if $M$ is free as a $\bD_{\ge r}$-module then $g_r(M)$ is equal to $\maxdeg(M/I_rM)$, where we write $\maxdeg(V)$ for the maximal degree occurring in a graded vector space $V$. Suppose that $N$ is a quotient of $M$. Then for any $r \ge 0$, we have a surjection $M/I_r M \to N/I_r N$, and so
\begin{displaymath}
\maxdeg(N/I_r N) \le \maxdeg(M/I_r M).
\end{displaymath}
Thus $g_r(N) \le g_r(M)$ whenever the two sides are defined, and the result follows. Now suppose $N$ is a submodule of $M$. Let $r$ be sufficiently large so that $M/N$ is free as a $\bD_{\ge r}$-module. Then the natural map $N/I_r N \to M/I_r M$ is injective, and so again one has the above inequality, and the result follows.
\end{proof}

\begin{corollary}
\label{d:lambda2}
Let $f \colon M \to N$ be a map of finitely presented $\bD$-modules. Then
\begin{displaymath}
\lambda(\ker(f)) \le \lambda(M), \qquad
\lambda(\im(f)) \le \lambda(M), \qquad
\lambda(\coker(f)) \le \lambda(N).
\end{displaymath}
\end{corollary}

%\begin{corollary}
%\label{d:lambda3}
%Suppose that $E^{p,q}_2 \implies H^{p+q}$ is a convergent first quadrant cohomological spectral sequence of finitely presented $\bD$-modules. Then
%\begin{displaymath}
%\lambda(H^n) \le \max_{p+q=n} \lambda(E^{p,q}_2)
%\end{displaymath}
%\end{corollary}
%
%\begin{proof}
%By the corollary, we have $\lambda(E^{p,q}_{k+1}) \le \lambda(E^{p,q}_k)$, and so $\lambda(E^{p,q}_{\infty}) \le \lambda(E^{p,q}_2)$. Since $H^n$ has a filtration where the graded pieces are $E^{p,q}_{\infty}$ with $p+q=n$, the result follows from Proposition~\ref{d:ext}.
%\end{proof}

For an integer $n$, we let $\ell(n)$ be the smallest non-negative integer $r$ such that $n \le p^r$. Note that $\ell(a+b) \le \max(\ell(a), \ell(b))+1$.

\begin{proposition}
\label{d:epsilon}
Let $f \colon M \to N$ be a map of finitely presented $\bD$-modules supported in non-negative degrees. Then
\begin{displaymath}
\epsilon(\ast) \le \max(\epsilon(M), \epsilon(N), \ell(\lambda(M)))+1,
\end{displaymath}
where $\ast$ is the kernel, cokernel, or image of $f$.
\end{proposition}

\begin{proof}
Let $r=\max(\epsilon(M), \epsilon(N))$, and let $s$ be the right side of the inequality in the statement of the proposition. Let $\{v_i\}$ and $\{w_j\}$ be bases for $M$ and $N$ as $\bD_{\ge r}$-modules. Then $f(v_i)=\sum_j a_{i,j} w_j$ for some $a_{i,j} \in \bD_{\ge r}$. Since the $w_j$ have non-negative degree, we find
\begin{displaymath}
\deg(a_{i,j}) \le \deg(v_i) < p^r+\lambda(M) \le p^s.
\end{displaymath}
Thus $a_{i,j}$ can only involve the variables $y_r, \ldots, y_{s-1}$, and so the kernel, cokernel, and image of $f$ are free over $\bD_{\ge s}$ by Proposition~\ref{d:free}.
\end{proof}

\begin{corollary}
	Let $I \subset \bD$ be a homogeneous ideal generated in degrees  $\le d$. Then \[ \maxdeg \Tor_1^{\bD}(\bk, I) \le d - 1 + p^{\ell(d) + 1}.\] 
\end{corollary}
\begin{proof} Let $F \to I$ be a surjection where $F$ is a free $\bD$-module generated in degrees $\le d$. Let $K$ be its kernel. Then we have $\lambda(K) \le d$ and $\epsilon(K) \le \ell(d) + 1$ (Proposition~\ref{d:free}). 
Keeping the notation $I_r$ of the proof of Proposition~\ref{d:lambda}, we see that $\Tor_1^{\bD}(\bk, I)$ is a subspace of $K/(I_0K)$ which, in turn, is a quotient of $K/(I_{\epsilon(K)}K)$. 	This shows that \[\maxdeg \Tor_1^{\bD}(\bk, I)  \le g_{\epsilon(K)} \le d  -1 + p^{\ell(d) + 1},\] completing the proof. 
\end{proof}

\section{Cohomology of $\FI$-modules}
\label{sec:FI-cohomology}

\subsection{The functor $\Gamma$}
\label{sec:gamma-generalities}

Write $\Mod_{\bk}^{\bN}$ for the category of graded $\bk$-modules supported in non-negative degrees. We define a functor
\begin{displaymath}
\Gamma \colon \Rep_{\bk}(S_{\ast}) \to \Mod_{\bk}^{\bN}, \qquad
\Gamma(V)_n = V_n^{S_n}.
\end{displaymath}
This functor is left-exact, and so we can consider its right derived functor $\rR \Gamma$. We write $\Gamma^t$ in place of $\rR^t \Gamma$. Since the functor $\Gamma$ is simply computed pointwise, so is $\Gamma^t$, and thus
\begin{displaymath}
\Gamma^t(V)_n = \rH^t(S_n, V_n).
\end{displaymath} 
One can define a dual version of $\Gamma$ by replacing invariants under the action of symmetric groups by coinvariants, but we don't pursue this here as homology groups can be studied in much greater generality; see \cite{rww}.  We now investigate some of further properties of $\Gamma$.

\begin{proposition}
	\label{prop:restriction-preserves-injectives} Let $\theta \colon \Mod_{\FI} \to \Rep_{\bk}(S_{\ast})$ be the restriction functor.  There is a canonical isomorphism  $\rR^t(\Gamma \circ \theta) \to \rR^t(\Gamma) \circ \theta$.
\end{proposition}
\begin{proof} Since $\theta$ preserves injectives ($\theta$ is right adjoint to an exact functor; see Proposition~\ref{prop:tensor-hom-adjunction}), we have a canonical isomorphism  $\rR(\Gamma \circ \theta) \to \rR(\Gamma) \circ \rR (\theta)$. The result now follows from the exactness of $\theta$. 
\end{proof}

From now on, we shall simply denote the composition $\Gamma \circ \theta$ by $\Gamma$. The previous proposition shows that this notation is reasonable.

\begin{proposition}
	\label{prop:tensor-functor}
The functor $\Gamma$ is naturally a (lax) symmetric monoidal functor, that is, for $S_{\ast}$-representations $V$ and $W$ there is a natural morphism $\Gamma(V) \otimes \Gamma(W) \to \Gamma(V \otimes W)$ satisfying the requisite axioms. If $\bk$ is a field then this morphism is an isomorphism, and so $\Gamma$ is a strict symmetric monoidal functor.
\end{proposition}

\begin{proof}
The asserted morphism is given by the following composite:
\begin{displaymath}
\begin{split}
(\Gamma(V) \otimes \Gamma(W))_n & = \bigoplus_{i+j=n} V_i^{S_i} \otimes W_j^{S_j} \to  \bigoplus_{i+j=n} (V_i \otimes W_j)^{S_i \times S_j} \\ &\to \bigoplus_{i+j=n} (\Ind_{S_i \times S_j}^{S_n}(V_i \otimes W_j))^{S_n} =\Gamma(V \otimes W)_n 
\end{split}
\end{displaymath} where the first arrow is the natural $\bk$-linear map  (which is an isomorphism if $\bk$ is a field), and the second arrow is the isomorphism given by Frobenius reciprocity. This morphism clearly respects the symmetries on both tensor product. This shows that $\Gamma$ is a symmetric monoidal functor. 
\end{proof}

%In fact, $\Gamma$ is a symmetric tensor functor, as the isomorphism of $\Gamma(V \otimes W)$ with $\Gamma(V) \otimes \Gamma(W)$ just constructed clearly respects the symmetries on both tensor products. 

The proposition above implies that $\Gamma$ takes commutative algebras in $\Rep_{\bk}(S_{\ast})$, that is, twisted commutative $\bk$-algebras, to commutative algebras in $\Mod_{\bk}^{\bN}$, that is, commutative graded $\bk$-algebras. In particular, $\Gamma(\bA)$ is a graded commutative $\bk$-algebra. The following proposition identifies it:

\begin{proposition}
The algebra $\Gamma(\bA)$ is the divided power algebra $\bD$.
\end{proposition}

\begin{proof}
The multiplication map $\Gamma(\bA) \otimes \Gamma(\bA) \to \Gamma(\bA)$ is the composite $ \Gamma(\bA) \otimes \Gamma(\bA) \to \Gamma(\bA \otimes \bA) \to \Gamma(\bA)$. Denote the multiplication map $\bA \otimes \bA \to \bA$ in degree $n$ by $\bigoplus_{i \le n} \Psi^i_n$. Here $\Psi_{n}^i \colon \Ind_{S_{i} \times S_{n-i}}^{S_{n}} \bA_i \otimes_{\bk} \bA_{n-i} \to \bA_{n}$ is the map induced by the multiplication map $\bA_i \otimes_{\bk} \bA_{n-i} \to \Res^{S_n}_{S_{i} \times S_{n-i}} \bA_{n}$ taking $a \otimes b$ to $ab$. Clearly, the images of the source and the target of $\Psi_n^i$ under $\Gamma$ can be naturally identified with $\bk$ and $\Gamma(\Psi_n^i)$ is the multiplication by the binomial coefficient $\binom{n}{i}$ under this identification. This shows that $\Gamma(\bA)$ is the divided power algebra $\bD$.  
\end{proof}

\begin{proposition}
\label{prop:composite-map}
Let $M$ be an $\FI$-module, and let $\Gamma^t(M) = \bigoplus_{n \ge 0} \rH^t(S_n, M_n)$ be the corresponding $\Gamma(\bA)$-module. Then the action of $x^{[m-n]} \in \Gamma(\bA)_{m-n}$ on $\Gamma^t(M)$ is given by the composition $\psi_{n,m}$ of the following maps:
\begin{displaymath}
\rH^t(S_n, M_n) \to \rH^t(S_n \times S_{m-n}, M_n) \to \rH^t(S_n \times S_{m-n}, M_m) \to \rH^t(S_m, M_m).
\end{displaymath}
The first map is pull-back along the group homomorphism $S_n \times S_{m-n} \to S_n$. The second map is induced by the $\FI$-module transition map $M_n \to M_m$. Finally, the last map is corestriction.
\end{proposition}
\begin{proof} Since all the maps in the composite are functorial in $M$, we may assume that $t=0$. Since $\Gamma$ is a symmetric monoidal functor, the multiplication $\Gamma(\bA)_{m-n} \otimes \Gamma(M)_n \to \Gamma(M)_m$ is given by the composite map $ \Gamma(\bA)_{m-n} \otimes \Gamma(M)_n \to \Gamma(\bA \otimes M)_n \to \Gamma(M)_n$. Thus the action of $x^{[m-n]}$ is given by the map $\Psi^{S_m}$ where $\Psi \colon \Ind_{S_{m-n} \times S_n}^{S_m} \bk \otimes_{\bk} M_n \to M_m$ is the natural map induced by the $\bA$-module structure on $M$. Clearly, $(\Ind_{S_{m-n} \times S_n}^{S_m} \bk \otimes_{\bk} M_n)^{S_m}$ and $(M_m)^{S_m}$ are naturally isomorphic to  $\Gamma(M)_m$ and $\Gamma(M)_n$ respectively, and $(\Psi)^{S_m}$ is the composite map in the assertion under this isomorphism. This completes the proof.
\end{proof}

\subsection{A connection on the cohomology of the symmetric group}
\label{sec:induced-connection}
In this section we reformulate a result due to Dold in terms of connections and generalize it to induced $\FI$-modules.  More precisely, we prove the following:

%Let $V$ be a $\bk[S_d]$-module $V$. Fix $t \ge 0$ and put $M=\Gamma^t(\cI(V))$. We have $M_n=\rH^t(S_n, \bk)$, and so restriction defines a map $M_n \to M_{n-1}$. These assemble to a single map $\nabla \colon M \to M[1]$.
%
%\begin{proposition}
%The map $\nabla$ is a connection.
%\end{proposition}
%
%\begin{proof}
%This follows immediately from \cite[Theorem~1]{D} -- Figure 12 in the proof there has the property that the composite map in the middle row is the sum of the ones in the top and the bottom row. The dual version of this is precisely the assertion of the proposition.
%\end{proof}
%
%\begin{corollary}
%We have $M \cong \bD \otimes_{\bk} M_0$ as $\bD$-modules, for some $\bk$-module $M_0$.
%\end{corollary}
%\begin{proof} Follows from Proposition~\ref{prop:connection}.
%\end{proof}
%
%The above corollary does not quite give Nakaoka's theorem since the module $M_0$ could apriori  be infinitely generated. However, that is the only way in which it is weaker than Nakaoka's theorem. More precisely, $M_0$ is easily seen (from \cite[Theorem~1 and Lemma~2]{D}) to be isomorphic to $\bigoplus_{r \ge 0} K_r[r]$ as in Proposition~\ref{prop:triv}, and Nakaoka's theorem also asserts the fact that $K_r = 0$ for $r > 2 t$.

%The following theorem encapsulates the key facts about the cohomology of symmetric groups that we need to prove Theorem~\ref{mainthm}. We follow the argument in } closely to prove the theorem.

\begin{proposition}\label{prop:induced-connection}
Let $V$ be an object of $\Rep_{\bk}(S_{\ast})$ and denote the $\bD$-module $\Gamma^t(\cI(V))$ by $M$. Then $M$ admits a connection $\nabla \colon M \to M[1]$. Moreover, if $V_n = 0$ for $n > d$ then we have the following: \begin{enumerate}
\item $\ker \nabla_n =0 $ for $n> 2t + d$.
\item $M \cong M_0 \otimes_{\bk} \bD$ for a graded $\bk$-module $M_0$ supported in degrees $\le 2 t +d$.
\end{enumerate}
\end{proposition}

The rest of \S \ref{sec:induced-connection} is devoted to proving this proposition. For this, we introduce some additional objects. Let $\Rep_{\bk}(S_{\ast})^G$ be the category of $G$-equivariant objects of $\Rep_{\bk}(S_{\ast})$: an object is a sequence $(M_n)_{n \ge 0}$ where $M_n$ is a $\bk[S_n \times G]$-module. For $M \in \Rep_{\bk}(S_{\ast})^G$, the module $\Gamma(M)$ has an action of $G$, and we put $\Gamma_G(M)=\Gamma(M)^G$. We let $\Gamma^t_G$ be the $t$-th  right derived functor of $\Gamma^t$. Thus for $M \in \Rep_{\bk}(S_{\ast})^G$ we have
\begin{displaymath}
\Gamma^t_G(M)_n = \rH^t(S_n \times G, M_n).
\end{displaymath}
Let $\Mod_{\bA}^G$ denote the category of $G$-equivariant $\bA$-modules, using the trivial action of $G$ on $\bA$.  For $M \in \Mod_{\bA}^G$, the $\bD$-module $\Gamma^t(M)$ has an action of $G$, and so $\Gamma^t_G(M)$ is also a $\bD$-module. The following result is a direct analog of Proposition~\ref{prop:composite-map}.

\begin{proposition}
\label{prop:Theta-D}
Let $M$  be an $G$-equivariant $\bA$-module. Then the action of $x^{[m-n]} \in \bD$ on $\Gamma^t_G(M)_n$ is given by the composition $\phi_{n,m}$ of the following maps:
{\small
\begin{displaymath}
\rH^t(S_n \times G, M_n) \to \rH^t(S_n \times S_{m-n} \times G, M_n) \to \rH^t(S_n \times S_{m-n} \times G, M_m) \to \rH^t(S_m \times G, M_m).
\end{displaymath}
}%
The first map is pull-back along the group homomorphism $S_n \times S_{m-n} \times G \to S_n \times G$. The second map is induced by the $\FI$-module transition map $M_n \to M_m$. Finally, the last map is corestriction.
\end{proposition}

The following lemma of Dold is crucial for our argument.

\begin{lemma}
\label{lem:dold:simple}
Let $M$ be a $\bk[S_{n+m} \times G]$-module. Then in the below diagram, the map defined by the middle path is the sum of the ones defined by the top and the bottom paths.
{\tiny
\begin{displaymath}
\begin{tikzcd}
 \rH^{t}(S_n \times S_{m -1} \times S_1 \times G, M) \ar[bend left=15, rightarrow]{rrd}{\cor} \\
 \rH^{t}(S_n \times S_{m} \times G, M) \ar[rightarrow]{r}{\cor}  \ar[rightarrow, swap]{u}{\res}  \ar[rightarrow]{d}{\res} & \rH^{t}(S_{n+m} \times G, M) \ar[rightarrow]{r}{\res} & \rH^{t}(S_{n+m-1}\times S_1 \times G, M)  \\
\rH^{t}(S_{n-1} \times S_1 \times S_{m} \times G, M)  \ar[rightarrow]{r}{\zeta^{\star}}  & \rH^{t}(S_{n-1} \times S_m \times S_1 \times G, M) \ar[bend right=15, rightarrow]{ru}{\cor} \\
\end{tikzcd}
%\caption{} \label{Fig:11} \ar[bend right=15, rightarrow, swap]{rru}{\cor \circ \; \zeta^{\star}}
\end{displaymath}
}%
Here $\zeta$ is the permutation defined by
\begin{displaymath}
(\zeta(1), \ldots, \zeta(n+m)) = (1,2, \ldots, n-1, n+m, n, n+1, \ldots, n+m-1),
\end{displaymath}
and $\res$ and $\cor$ are restriction and corestriction respectively.
\end{lemma}

\begin{proof}
Let $P$ be a $\bk[S_{n+m} \times G]$-free resolution of $\bk$. Then $P$ is an $H$-free resolution for every subgroup $H$ of $S_{n+m} \times G$. To keep Dold's (\cite{D}) notation, we denote the cochain complex $\Hom_{\bk[G]}(P, M)$ by $A$. There is an action of $S_{n+m}$ on $A$ given by $(\sigma f)(p) = \sigma (f(\sigma^{-1}p))$, and for any subgroup $H$ of $S_{n+m}$, we have $\rH^{t}(H \times G, M) = \rH^t(A^{H})$. With this observation in mind, the assertion follows by applying \cite[Lemma~1, dual version]{D} to this $A$.
\end{proof}

\begin{proposition} \label{prop:connection-on-J}
Let $V$ be a $\bk[G]$-module, regarded as an object of $\Rep_{\bk}(S_{\ast})^G$ concentrated in degree~0. The restriction map
\begin{displaymath}
\rH^t(S_n \times G,  V) \to \rH^t(S_{n-1} \times G,  V)
\end{displaymath}
defines a connection $\nabla$ on the $\bD$-module $\Gamma_G^t(\cI(V))$. (Note here that $\cI(V)_n$ is the $\bk[S_n \times G]$-module $V$ with $S_n$ acting trivially.) 
\end{proposition}

\begin{proof}
Consider the following diagram:
{\tiny
\begin{displaymath}
\begin{tikzcd}
\rH^t(S_n \times G, V)  \ar{r}{\res}    & \rH^{t}(S_n \times S_{m -1} \times S_1 \times G, V) \ar[bend left=15]{rrd}{\cor} \\
{\rH^t(S_n \times G, V)}  \ar{u}{\id } \ar[swap]{d}{\res} \ar{r}{\res} & \rH^{t}(S_n \times S_{m} \times G, V) \ar{r}{\cor}  \ar[swap]{u}{\res}  \ar{d}{\res} & \rH^{t}(S_{n+m} \times G, V) \ar{r}{\res} & \rH^{t}(S_{n+m-1}\times S_1 \times G, V)  \\
\rH^t(S_{n-1} \times S_1 \times G, V)  \ar{r}{\res} & \rH^{t}(S_{n-1} \times S_1 \times S_{m} \times G, V) \ar[bend right=15, swap]{rru}{\cor \circ \; \zeta^{\star}} \\
\end{tikzcd}
\end{displaymath}
}%
By Proposition~\ref{prop:Theta-D}, the composite maps defined by the top, the middle and the bottom paths are $x^{[m-1]}$, $\nabla x^{[m]}$ and $x^{[m]} \nabla$ respectively. Since restriction is functorial, we see that the top left and bottom left squares commute. Thus, by the previous lemma, we have $\nabla x^{[m]} = x^{[m]} \nabla + d(x^{[m]})$. This completes the proof.
\end{proof}

We shall need a group theoretic analog of the binomial identity \[ \binom{n+m}{n} \binom{n}{d} = \binom{n+m}{d} \binom{n+m-d}{n-d}. \] We introduce a notation first. Let $B \subset A$ be subsets of a finite ordered set $U$. We can think of $A,B$ and $U$ as objects of $\FI$ once we forget the order. Suppose $S \in \binom{A}{B}$, that is, $S$ is a subset of $A$ of size $|B|$. We denote, by $\gamma_S$, the element of $\Aut_{\FI}(U)$ that fixes the complement of $A$ pointwise, takes $B$ to $S$, and is order preserving on both $B$ and $A \setminus B$. The following lemma is easy to verify and provides a group theoretic analog of the identity above.

\begin{lemma} Let $\{A_1, A_2, A_3\}$ be  a partition of a finite ordered set $U$. Then the following identity holds in $\bk[\Aut_{\FI}(U)]$ 
	\[ \sum_{S' \in \binom{U}{A_1 \sqcup A_2}} \sum_{S \in \binom{A_1 \sqcup A_2 }{A_1}} \gamma_{S'}\gamma_S = \sum_{T' \in \binom{U}{A_2 \sqcup A_3}} \sum_{T \in \binom{A_2 \sqcup A_3}{A_2}} \gamma_{T'}\gamma_T.  \] Moreover, all the elements of $\Aut_{\FI}(U)$ appearing in the left (right) hand side are distinct.
\end{lemma}

\begin{proposition}
\label{prop:shapiro}
Let $V$ be an object of $\Rep_{\bk}(S_{\ast})$ concentrated in degree~$d$, and let $V'=V_d$, regarded as an object of $\Rep_{\bk}(S_{\ast})^{S_d}$ concentrated in degree~0. Then Shapiro isomorphism $\Gamma^t(\cI(V)) \to \Gamma_{S_d}^t(\cI(V'))[d]$ is a map of $\bD$-modules.  
\end{proposition}

\begin{proof}
Let $f \colon [n] \to [n+m]$ be the natural inclusion. We need to verify the commutativity  of the following diagram.
\begingroup
\setlength{\belowcaptionskip}{1000pt}
\fontsize{8.0pt}{12pt}\selectfont
\begin{displaymath}
\begin{tikzcd}
\rH^{t}(S_n, \cI(V)_n) \ar{r}{\res}    & \rH^{t}(S_{n} \times S_m, \cI(V)_{n})  \ar{r}{f_{\star}} & \rH^{t}(S_{n} \times S_m, \cI(V)_{n+m})  \ar{r}{\cor}& \rH^{t}(S_{n+m}, \cI(V)_{n+m}) \\
\rH^{t}(S_{n-d} \times S_d, V) \ar{r}{ \res}  \ar{u}{\text{Shapiro}}  & \rH^{t}(S_{n-d} \times S_m \times S_d, V)  \ar{r}{\id}  & \rH^{t}(S_{n-d} \times S_m \times S_d, V) \ar{r}{\cor} & \rH^{t}(S_{n+m-d} \times S_d, V) \ar[swap]{u}{\text{Shapiro}} 
\end{tikzcd}
\end{displaymath}
\endgroup
 Each of the cohomology groups in the diagram above can be calculated by using a common $\bk[S_{n+m}]$-free resolution of $\bk$. Thus is suffices to verify the commutativity of the following diagram for a given $\bk[S_{n+m}]$-module $P$. \begingroup
\setlength{\belowcaptionskip}{1000pt}
\fontsize{8.0pt}{12pt}\selectfont
\begin{displaymath}
\begin{tikzcd}
\Hom_{S_n}(P, \cI(V)_n) \ar{r}{\res}     & \Hom_{S_{n} \times S_m}(P, \cI(V)_{n})  \ar{r}{f_{\star}} & \Hom_{S_{n} \times S_m}(P, \cI(V)_{n+m})  \ar{r}{\cor}& \Hom_{S_{n+m}}(P, \cI(V)_{n+m})  \\
\Hom_{S_{n-d} \times S_d}(P, V) \ar{r}{ \res} \ar{u}{Shapiro} & \Hom_{S_{n-d} \times S_m \times S_d}(P, V)  \ar{r}{\id}  & \Hom_{S_{n-d} \times S_m \times S_d}(P, V) \ar{r}{\cor} & \Hom_{S_{n+m-d} \times S_d}(P, V) \ar[swap]{u}{Shapiro}
\end{tikzcd}
\end{displaymath} 
\endgroup Set $U = [n+m]$, $A_1 = [d]$, $A_2 = [n] \setminus [d]$ and $A_3 = [n+m] \setminus [n]$.  Then we have  \[\cI(V)_{n+m} = \bk[\Aut_{\FI}(U)] \otimes_{\bk[\Aut_{\FI}(A_2 \sqcup A_3) \times \Aut_{\FI}(A_1)]} V.\] Fix an element $a \in \Hom_{S_{n-d} \times S_d}(P, V)$. The images, say $\phi$ and $\psi$, of $a$ in $\Hom_{S_{n+m}}(P, \cI(V)_{n+m})$ along the top and the bottom paths are given by  \begin{align*}
\phi(p) &= \sum_{S' \in \binom{U}{A_1 \sqcup A_2}} \sum_{S \in \binom{A_1 \sqcup A_2 }{A_1}} \gamma_{S'}\gamma_S \otimes a(\gamma_S^{-1} \gamma_{S'}^{-1} p) \\
\psi(p) &= \sum_{T' \in \binom{U}{A_2 \sqcup A_3}} \sum_{T \in \binom{A_2 \sqcup A_3}{A_2}} \gamma_{T'} \otimes \gamma_T a(\gamma_T^{-1} \gamma_{T'}^{-1} p) \\
&= \sum_{T' \in \binom{U}{A_2 \sqcup A_3}} \sum_{T \in \binom{A_2 \sqcup A_3}{A_2}} \gamma_{T'} \gamma_T \otimes  a(\gamma_T^{-1} \gamma_{T'}^{-1} p)
\end{align*} Thus by the previous lemma, we see that $\phi = \psi$. This completes the proof.
\end{proof}

\begin{lemma}
\label{lem:bootstrap}
Suppose $n > 2 t \ge 0$. Then the restriction map $\rH^{t+1}(S_n, \bZ) \to \rH^{t+1}(S_{n-1}, \bZ)$ is an isomorphism.
\end{lemma}
\begin{proof} Nakaoka's stability theorem  (\cite{nakaoka}) for the homology of symmetric groups states that the restriction map $\rH_{t}(S_{n-1}, \bZ) \to \rH_t(S_{n}, \bZ)$ is an isomorphism for $n >2 t$. Thus the result follows from the isomorphism  $\rH^{t+1}(S_n, \bZ) \cong \ext^1_{\bZ}(\rH_t(S_n, \bZ), \bZ)$ (deduced from the universal coefficient theorem).
\end{proof}

\begin{proof}[Proof of Proposition~\ref{prop:induced-connection}]
It suffices to treat the case where $V$ is concentrated in degree $d$. Let $V'=V_d$, regarded as an object of $\Rep_{\bk}(S_{\ast})^{S_d}$ in degree~0. Set $M = \Gamma_{S_d}^t(\cI(V'))$. By Proposition~\ref{prop:shapiro}, $\Gamma^t(\cI(V))$ is isomorphic to $M[d]$ as $\bD$-modules, and by Proposition~\ref{prop:connection-on-J} we have a connection $\nabla$ on $M$.

We claim that $\nabla_n$ is an isomorphism for $n>2t$. To see this, observe that (see \cite{mathoverflow})
{\small
\begin{align*}
M_n = \rH^t(S_d \times S_n, V \boxtimes \bZ) &=  \bigoplus_{a + b =t+1} \Tor_1^{\bZ}(\rH^a(S_d, V), \rH^b(S_{n}, \bZ)) \oplus \bigoplus_{a + b =t} \rH^a(S_d, V) \otimes \rH^b(S_{n}, \bZ) 
\end{align*}
}%
By Lemma~\ref{lem:bootstrap}, the restriction map $\rH^b(S_{n}, \bZ) \to \rH^b(S_{n-1}, \bZ)$ is an isomorphism for $n >2t$. Since the K\"{unneth} formula commutes with restriction, the claim follows.

The above claim implies $\ker(\nabla_n)=0$ for $n>2t$, which proves statement (a). By Proposition~\ref{prop:connection}, we have $M/\bD_{+} M = \ker \nabla$ and an isomorphism $M \cong \bD \otimes_{\bk} \ker(\nabla)$, and so (b) follows.
\end{proof}

\subsection{Proof of main theorem}
\label{sec:proof-main}
 For an $\FI$-module $M$, define $\ul{\Gamma}(M)=\Gamma(\bS(M))$, where $\bS$ is the saturation functor. The functor $\ul{\Gamma}$ is left-exact, so we can consider its right derived functors $\rR \ul{\Gamma}$. We put $\ul{\Gamma}^t=\rR^t \ul{\Gamma}$. We note that the map $M \to \bS(M)$ induces a map $\Gamma(M) \to \ul{\Gamma}(M)$, and thus maps $\Gamma^t(M) \to \ul{\Gamma}^t(M)$ for all $t \ge 0$.

\begin{theorem}
\label{thm:main}
Let $M$ be a finitely generated $\FI$-module. Then for each $t \ge 0$: 
\begin{enumerate}
\item $\ul{\Gamma}^t(M)$ is a finitely presented $\bD$-module.
\item The map $\Gamma^t(M) \to \ul{\Gamma}^t(M)$ is an isomorphism in degrees $\ge g(M)+\max(g(M), r(M))$.
\item If $\bk$ is a field of characteristic $p>0$, then $\lambda(\ul{\Gamma}^t(M)) \le 2t+\delta(M)$.
\end{enumerate}
\end{theorem}

We need a lemma. 

\begin{lemma}
	The saturation functor $\bS$ preserves injectives. In particular, we have a canonical isomorphism $\rR \ul{\Gamma} \to \rR\Gamma \circ \rR \bS$.
\end{lemma}
\begin{proof}
	By definition, $\bS = \rS \circ \rT$. Since $\rS$ is right adjoint to an exact functor, it preserves injectives. It suffices to show that the localization functor $\rT$ preserves injectives. Suppose $M$ is a torsion $\FI$-module and $I$ is its injective envelope in $\Mod_{\FI}$. Then $I$ is an essential extension of $M$. Since an essential extension of a torsion $\FI$-module is itself a torsion $\FI$-module, we conclude that $I$ is a torsion $\FI$-module. This shows that $\rT$ preserves injectives (\cite[p.375, Corollaire~3]{gabriel}). The second assertion is now immediate.
\end{proof}

\begin{proof}[Proof of Theorem~\ref{thm:main}]
(a) If $M$ is induced then $\Gamma^t(M)$ is finitely presented by Proposition~\ref{prop:induced-connection}. If $M$ is semi-induced, then $\Gamma^t(M)$ is finitely presented by d\'evissage to the induced case ($\bD$ is coherent). Finally, for a general module $M$, consider a complex $M \to I^{\bullet}$ as in Theorem~\ref{thm:semires}, so that $\rR \bS(M)\cong I^{\bullet}$. Then $\rR \ul{\Gamma}(M) \cong \rR \Gamma(I^{\bullet})$ (by the previous lemma), and so there is a spectral sequence
\begin{displaymath}
\Gamma^i(I^j) \implies \ul{\Gamma}^{i+j}(M),
\end{displaymath}
from which it follows that $\ul{\Gamma}^t(M)$ is finitely presented.

(b) The complex $M \to I^{\bullet}$ of Theorem~\ref{thm:semires} is exact in degrees $\ge g(M)+\max(g(M), r(M))$ (Theorem~\ref{thm:homology-bound}), and so it follows that the map $\Gamma^t(M) \to \Gamma^t(I^{\bullet})\cong \ul{\Gamma}^t(M)$ is an isomorphism in degrees $\ge g(M)+\max(g(M), r(M))$, since $\Gamma^t$ is computed degree-wise.

(c) First, suppose $M$ is induced from degree $\le d$. Then Proposition~\ref{prop:induced-connection} shows that $\Gamma^t(M)$ is a free $\bD$-module finitely generated in degrees $\le 2t+d$, and so $\lambda(\Gamma^t(M)) \le 2t+d$.

Now suppose that $M$ is semi-induced of degree $\le d$. We prove that $\lambda(\Gamma^t(M)) \le 2t+d$ by induction on the minimal length of a filtration of $M$ with induced subquotients. The base case holds by the previous sentence. In general, there is a short exact sequence
\begin{displaymath}
0 \to N \to M \to N' \to 0
\end{displaymath}
where $N,N'$ are semi-induced of degree $\le d$ and admit filtrations of smaller length than that of $M$. We obtain an exact sequence
\begin{displaymath}
\Gamma^t(N) \to \Gamma^t(M) \to \Gamma^t(N')
\end{displaymath}
By induction, we can assume $\lambda(\Gamma^t(N)), \lambda(\Gamma^t(N')) \le 2t+d$. It follows from Propositions~\ref{d:ext} and~\ref{d:lambda} that $\lambda(\Gamma^t(M)) \le 2t+d$. 

Next suppose that $I^{\bullet}$ is a finite length complex of semi-induced modules with $I^i$ of degree $\le d-i$. We will show $\lambda(\Gamma^t(I^{\bullet})) \le 2t+d$. Let $J^{\bullet}$ be the complex $I^1 \to I^2 \to \cdots$. We have a short exact sequence of complexes
\begin{displaymath}
0 \to J^{\bullet}[1] \to I^{\bullet} \to I^0 \to 0.
\end{displaymath}
We thus obtain an exact sequence
\begin{displaymath}
\Gamma^{t-1}(J^{\bullet}) \to \Gamma^t(I^{\bullet}) \to \Gamma^t(I^0).
\end{displaymath}
By induction on $d$, we can assume $\lambda(\Gamma^{t-1}(J^{\bullet})) \le 2(t-1)+(d-1)$, and the previous paragraph shows $\lambda(\Gamma^t(I^0)) \le 2t+d$. Once again, we conclude $\Gamma^t(I^{\bullet}) \le 2t+d$.

Finally, let $M$ be a finitely generated $\FI$-module of degree $\le d$. Let $M \to I^{\bullet}$ be as in Theorem~\ref{thm:semires}. Then $\ul{\Gamma}^t(M)=\Gamma^t(I^{\bullet})$, and so $\lambda(\ul{\Gamma}^t(M)) \le 2t+\delta(M)$ by the previous paragraph.
\end{proof}

\begin{corollary}
If $M$ is a finitely generated $\FI$-module then $\Gamma^t(M)$ is nearly finitely presented.
\end{corollary}

\begin{corollary} \label{cor:pd}
Let $M$ be a finitely generated $\FI$-module over a field $\bk$ of characteristic $p$, and let $t \ge 0$ be given. Then there exists a power $q$ of $p$ such that
\begin{displaymath}
\dim_{\bk} \rH^t(S_n, M_n) = \dim_{\bk} \rH^t(S_{n+q}, M_{n+q})
\end{displaymath}
holds for all $n \ge \max(g(M)+r(M), 2g(M), 2t+\delta(M))$.
\end{corollary}

\begin{example}
\label{ex:nfp} $\Gamma^t(M)$ may not be finitely presented: assume that $t=0$ and $M$ be the $\FI$-module over $\bF_p$ satisfying  $M_0 = \bF_p$ and $M_n = 0$ for $n>0$. Then we have $\Gamma^0(M) = \bD/\bD_{+}$ which is clearly not finitely presented ($\bD_{+}$ is not finitely generated).  
\end{example}

\subsection{Bounds on the period}

In this section, $\bk$ is a field of characteristic $p$. Our goal is to prove the following result:

\begin{theorem}
\label{thm:main-positive}
Let $M$ be a finitely generated $\FI$-module and suppose $\delta(M) < q$, with $q=p^r$. Then the $\bD$-module $\ul{\Gamma}^t(M)$ admits a $q$-connection. In particular, $\epsilon(\ul{\Gamma}^t(M)) \le r$.
\end{theorem}

\begin{corollary}
In Corollary~\ref{cor:pd}, one can take $q$ to be any power of $p$ such that $\delta(M) < q$.
\end{corollary}

We require a number of preparatory results. In this section, $\rR^{\bullet} \Gamma(V)$ denotes $\bigoplus_{t \ge 0} \Gamma^t(V)$. Note that this space is bigraded, since each $\Gamma^t(V)$ is itself a graded vector space. We call $t$ the cohomological index; writing $\Gamma(V)=\bigoplus_{n \ge 0} \Gamma(V_n)$, we call $n$ the $\FI$-index.

\begin{proposition}
	\label{prop:total-tensor-functor}
The functor $\rR^{\bullet} \Gamma$ respects tensor products, that is, if $M$ and $N$ are representations of $S_{\ast}$ then there is a natural symmetric monoidal isomorphism \[\rR^{\bullet} \Gamma(M \otimes N) \xrightarrow{\cong} \rR^{\bullet} \Gamma(M) \otimes \rR^{\bullet} \Gamma(N).\] Furthermore, this isomorphism is compatible with the symmetries of the tensor products (using the appropriate sign rule on the cohomological index).
\end{proposition}

\begin{proof}
Since $\rR \Gamma$ and $\otimes$ commute with arbitrary direct sums, it suffices to treat the case where $M$ is a concentrated in degree $m$ and $N$ is concentrated in degree $n$. The result then follows from Shapiro's lemma and the K\"unneth formula:
\begin{displaymath}
\begin{split}
\rR^{\bullet} \Gamma(M \otimes N)
&= \rH^{\bullet}(S_{m+n}, \Ind_{S_m \times S_n}^{S_{n+m}}(M_m \otimes N_n))
\xrightarrow{\cong} \rH^{\bullet}(S_m \times S_n, M_m \otimes N_n) \\
&\xrightarrow{\cong} \rH^{\bullet}(S_m, M_m) \otimes \rH^{\bullet}(S_n, N_n)
= \rR^{\bullet} \Gamma(M) \otimes \rR^{\bullet} \Gamma(N)
\end{split}
\end{displaymath}
The symmetry and associativity properties are standard.
\end{proof}

In particular, we see that $\bE=\rR^{\bullet} \Gamma(\bA)$ naturally has the structure of an algebra. Note that $\bE^t_n=\rH^t(S_n, \bk)$. Let $\bd \colon \bE^t_n \to \bE^t_{n-1}$ be the map induced by restricting cohomology classes from $S_n$ to $S_{n-1}$.

\begin{proposition}
The map $\bd \colon \bE \to \bE$ is a derivation.
\end{proposition}

\begin{proof}
We need to show that in the diagram below the, middle row is the sum of the top and the bottom rows (where $\cup$ is the cup product and rest of the notation are as in Lemma~\ref{lem:dold:simple}):

\begingroup
\setlength{\belowcaptionskip}{1000pt}
   \fontsize{8.0pt}{12pt}\selectfont
\begin{displaymath}
\begin{tikzcd}
\rH^a(S_n, \bk) \otimes \rH^b(S_{m-1}, \bk) \ar{r}{\cup}    & \rH^{a+b}(S_n \times S_{m -1} \times S_1, \bk) \ar[bend left=15, dashrightarrow]{rrd}{\cor} \\
{\rH^a(S_n, \bk) \otimes \rH^b(S_{m}, \bk)} \ar[bend left = 15]{u}{\id \otimes \res} \ar[bend right = 15, swap]{d}{\res \otimes \id} \ar{r}{\cup} & \rH^{a+b}(S_n \times S_{m}, \bk) \ar[dashrightarrow]{r}{\cor}  \ar[dashrightarrow, swap]{u}{\res}  \ar[dashrightarrow]{d}{\res} & \rH^{a+b}(S_{n+m}, \bk) \ar[dashrightarrow]{r}{\res} & \rH^{a+b}(S_{n+m-1}\times S_1, \bk)  \\
\rH^a(S_{n-1}, \bk) \otimes \rH^b(S_{m}, \bk)   \ar{r}{\cup} & \rH^{a+b}(S_{n-1} \times S_1 \times S_{m}, \bk) \ar[bend right=15, dashrightarrow, swap]{rru}{\cor \circ \; \zeta^{\star}} \\
\end{tikzcd}
%\caption{} \label{Fig:11}
\end{displaymath}
\endgroup By Lemma~\ref{lem:dold:simple}, the dashed part has the property that the middle row is the sum of the top and the bottom rows. Also, the cup product commutes with restriction. Thus the top-left and the bottom-left squares commute. The assertion now follows. 
\end{proof}
Let $V$ be a representation of $S_*$. We constructed a connection $\nabla$ on the $\bD$-module $\Gamma^t(\cI(V))$. Write $\nabla$ still for the induced map on $\rR^{\bullet} \Gamma(\cI(V))$. Note that Proposition~\ref{prop:total-tensor-functor} implies that there is a natural isomorphism $\rR^{\bullet} \Gamma(\cI(V)) \cong \rR^{\bullet} \Gamma(V) \otimes \bE$ of $\bE$-modules.

\begin{proposition}
	\label{prop:compatible-connections}
Under the identification $\rR^{\bullet} \Gamma(\cI(V)) = \rR^{\bullet} \Gamma(V) \otimes \bE$, the map $\nabla$ corresponds to $1 \otimes \bd$.
\end{proposition}

\begin{proof}
It suffices to consider the case when $V$ is concentrated in degree $d$. With the notation of Proposition~\ref{prop:shapiro}, the identification is given by the composite
\begin{displaymath}
\rR^{\bullet} \Gamma(\cI(V))  \to \rR^{\bullet}\Gamma_{S_d}(\cI(V'))[d] \to \rR^{\bullet} \Gamma(V) \otimes \bE
\end{displaymath}
where the first map is the Shapiro isomorphism and the second map is the K\"unneth isomorphism. Recall that $\nabla$ is defined to correspond (via the Shapiro isomorphism) to the connection on $\rR^{\bullet}\Gamma_{S_d}(\cI(V'))[d]$, and the connection on $\rR^{\bullet}\Gamma_{S_d}(\cI(V'))[d]$ is given in degree $n$ by the restriction map (Proposition~\ref{prop:connection-on-J}) \[\rH^t(S_{n-d} \times S_d,  V) \to \rH^t(S_{n-d-1} \times S_d,  V). \] Since the K\"unneth isomorphism  commutes with restriction, we see that $\nabla$ corresponds to $1 \otimes \bd$. 
\end{proof}

\begin{proposition} \label{prop:compat-conn}
Let $V$ and $W$ be representations of $S_{\ast}$ and let $f \colon \cI(V) \to \cI(W)$ be a map of $\bA$-modules. Suppose $V$ is supported in degrees $<q$, where $q$ is a power of $p$. Then the map $\Gamma^t(\cI(V)) \to \Gamma^t(\cI(W))$ induced by $f$ is compatible with the $q$-connection $\nabla^q$, for all $t \ge 0$.
\end{proposition}

\begin{proof}
The map $\rR^{\bullet} \Gamma(\cI(V)) \to \rR^{\bullet} \Gamma(\cI(W))$ induced by $f$ is one of $\bE$-modules. Note that $\rR^{\bullet} \Gamma(\cI(V))$ is isomorphic to $\rR^{\bullet} \Gamma(V) \otimes \bE$ and $\nabla^q$ is compatible with the $q$-fold iterate of $1 \otimes \bd$ (Proposition~\ref{prop:compatible-connections}), and an analogous statement holds for the target. Now suppose $a \in \rR^{\bullet} \Gamma(V)$ and $f(a) = \sum_i b_i \otimes e_i$ with $e_i \in \bE$. By our choice of $q$, we can assume that the degree of each $e_i$ is smaller than $q$. In particular, $\bd^q(e_i) =0$.  Thus for any $e \in \bE$,  we have \[\nabla^q f(a \otimes e) = \sum_i  b_i \otimes \bd^q(e_i e) = \sum_i  b_i \otimes e_i \bd^q( e) = f(a ) \otimes \bd^q( e) = f(\nabla^q(a \otimes e)).\]  This completes the proof.
\end{proof}

%, it suffices to show that $f$ carries $\rR^{\bullet} \Gamma(V)$ into the kernel of $\nabla^q$. But this is automatic: every element of $\rR^{\bullet} \Gamma(V)$ has $\FI$-degree $<q$, and so the same is true for the elements of the image, and they are thus annihilated by $\nabla^q$ for degree reasons.

\begin{remark}
Note that the statement of the proposition is about a single cohomology group, but the proof (which is the simplest one we know) makes use of the $\bE$-module structure on all cohomology groups.
\end{remark}

\begin{proof}[Proof of Theorem~\ref{thm:main-positive}]
By Corollary~\ref{cor:derived-saturation}, $\rR\bS(M)$ is quasi-isomorphic to a bounded complex $I^{\bullet}$ of semi-induced modules generated in degree $\le \delta(M)$. By Proposition~\ref{prop:induced-resolution} (and the following discussion), $I^{\bullet}$ is quasi-isomorphic to a bounded complex of induced modules generated in degree $\le \delta(M)$. The result now follows from Proposition~\ref{prop:compat-conn} and Proposition~\ref{prop:q-connection}.
\end{proof}

\section{Applications}
\label{sec:applications}

\subsection{Periodicity in the cohomology of Specht modules} 
\label{sec:specht}

Assume that $\bk$ is a field of characteristic $p$. Recall that $\bM_{\mu}$ is the Specht module corresponding to a partition $\mu$ and that $\mu[n]$ is the partition $(n -|\mu|, \mu)$ (defined only for $n \ge \mu_1 + |\mu|$). We have the following:

%Hemmer proved in \cite[Theorem~7.1.8]{hemmer} a periodicity result on the first cohomology groups $\rH^1(S_n, \bM_{\mu[n]})$. We generalize it to higher cohomology groups. More precisely, we have the following:

\begin{theorem}
\label{thm:hemmer}
Let $\mu$ be a partition of $d$ and $q$ be a power of $p$ strictly larger than $d$. If $n \ge \max(2 d +\mu_1, 2 t + d)$, then we have
 \[\rH^t(S_n, \bM_{\mu[n]}) = \rH^t(S_{n+q}, \bM_{\mu[n+q]}).\]  
\end{theorem}

We start with a proposition.
\begin{proposition} 
\label{prop:L-lambda}
Suppose $\mu$ is a partition of $d$. There is an $\FI$-module $L_{\mu}$ such that the following hold: \begin{enumerate}
\item $(L_{\mu})_n = \bM_{\mu[n]}$
\item $\delta(L_{\mu}) = d$ 
\item $g(L_{\mu}) = d + \mu_1$
\item $r(L_{\mu}) \le 2 d + \mu_1 + 1$
\item There exists a finite length complex $0 \to L_{\mu} \to I^{\bullet}$ exact in degrees $\ge 2 d + \mu_1 $ where, for each $i$, $I^{i}$ is a semi-induced module with $\delta(I^i) \le d$.
\end{enumerate}
\end{proposition}
\begin{proof} We first construct $L_{\mu}$. Let $V_{\mu}$ be the permutation module of $S_d$ corresponding to the partition $\mu$. Now note that for $n \ge d + \mu_1$,  $\cI(V_{\mu})_n$ is naturally a permutation module for the partition $\mu[n]$. For each $n\ge d + \mu_1$, inductively pick a Young tableau $T_n$ of shape $\mu[n]$ with entries in $[n]$ such that $T_n$ and $T_m$ agree on $\mu[d+\mu_1]$. We can think of $\cI(V_{\mu})_n$ as the $\bk[S_n]$-module generated by $T_n$. Let $C_{\mu[n]}$ be the column subgroup corresponding to $T_n$. By construction, we have $C_n = i_{d+\mu_1, n} C_{d+\mu_1}$ where $i_{m,n} \colon [m] \to [n]$ is the natural inclusion. Thus if $e = \sum_{\sigma \in C_{d + \mu_1}} \sgn(\sigma) \sigma$ then we have \[\bM_{\mu[n]} = (i_{d+\mu_1, n} e) T_n = i_{d+\mu_1, n}(e T_{d + \mu_1}) = i_{d+\mu_1, n} \bM_{d + \mu_1}. \] This shows that the sub-$\FI$-module of $\cI(V_{\mu})$ generated in degree $d + \mu_1$ by $e T_{d + \mu_1}$ is isomorphic to $\bM_{\mu[n]}$ in every degree $n \ge d + \mu_1$. Define $L_{\mu}$ to be this sub-$\FI$-module. This already proves (a) and (c). Since $L_{\mu}$ is a submodule of $\cI(V_{\mu})$, we have $\delta(L_{\mu}) \le \delta(\cI(V_{\mu}))	=d$. Note that since $\cI(V_{\mu}) \to \Sigma^n \cI(V_{\mu})$ admits a section, the inclusion $\iota \colon L_{\mu} \to \cI(V_{\mu})$ factors through $\Sigma^n L_{\mu}$. Thus if $\delta(L_{\mu}) < d$, then for large $n$, $\Sigma^n L_{\mu}$ is generated in degree $<d$ (follows from Theorem~\ref{thm:shift}) which implies $\iota =0$, a contradiction. This proves (b).  Now consider the exact sequence \[ 0 \to L_{\mu} \to \cI(V_{\mu}) \to K \to 0. \]  We have $g(K) \le d$ and $r(K) \le g(L_{\mu}) = d + \mu_1$.  So by Theorem~\ref{thm:homology-bound}, $K$ admits a complex as in Theorem~\ref{thm:semires} exact in degrees $\ge 2 d + \mu_1$. (e) now follows from the exact sequence above. In \cite[Theorem~A]{castelnuovo-regularity}, it is proven that Castelnuovo-regularity of an $\FI$-modules $M$ is at most $g(M) + r(M) -1$. It follows that $r(L_{\mu}) = \deg \Tor^1_{\bA}(L_{\mu}, \bk) = \deg \Tor^2_{\bA}(K, \bk) \le g(K) + r(K) + 1$. Thus $r(L_{\mu}) \le 2 d + \mu_1 + 1$, completing the proof.
\end{proof}

\begin{proof}[Proof of Theorem~\ref{thm:hemmer}] Let $I^{\bullet}$ be the complex in Proposition~\ref{prop:L-lambda}(e). Then we have $\ul{\Gamma}^t(L_{\mu}) = \Gamma^t(I^{\bullet})$, and the natural map $\Gamma^t(L_{\mu}) \to \ul{\Gamma}^t(L_{\mu})$ is an isomorphism in degrees $\ge 2 d +\mu_1$. Now by Proposition~\ref{prop:compat-conn}, $\Gamma^t(I^{\bullet})$ admits a $q$-connection. We also have $\lambda(\Gamma^t(I^{\bullet})) \le 2 t + d$ by Theorem~\ref{thm:main}(c). Thus by Proposition~\ref{prop:stability-period}, we see that $\Gamma^t(I^{\bullet})_n = \Gamma^t(I^{\bullet})_{n + q}$ for $n \ge 2t +d$. The result now follows because $\Gamma^t(I^{\bullet})_n = \Gamma^t(L_{\mu})_n = \rH^t(S_n, \bM_{\mu[n]})$ for $n  \ge 2 d + \mu_1$.
\end{proof}

\subsection{Integral cohomology of unordered configuration spaces}
\label{sec:config}
Let $\cM$ be a manifold. The unordered configuration space $\uconf_n(\cM)$ is  given by \[ \uconf_n(\cM) \coloneq \{ (P_1, P_2, \ldots, P_n) \in \cM^n \colon P_i \neq P_j  \}/S_n.   \] Some recent results have shown that under certain mild assumption on $\cM$, the cohomology groups $\rH^t(\uconf_n(M), \bF_p)$ are periodic (see  \cite{jeremy} for the latest result). We have an integral version in this direction: 

\begin{theorem}
	\label{thm:configuration-spaces}
Suppose $\bk$ is a commutative noetherian ring and fix a $t\ge 0$. Let $\cM$ be a connected manifold of dimension $\ge 2$ with the homotopy type of a finite CW complex. Then there exists a finitely presented $\bD$-module $M$ such that \[\rH^t(\uconf_n(\cM), \bk) \cong M_n\] for $n \gg 0$. 
\end{theorem}
\begin{proof}
In \cite[\S 4]{nagpal}, it is shown that there is a bounded below complex $\cU^{\bullet}$ of finitely generated $\FI$-modules such that $\Gamma^s(\cU^{\bullet})_n = \rH^s(\uconf_n(\cM), \bk)$ for all $s$ (The proof in the reference is only given when $\bk$ is a field of positive characteristic but it works without change for any noetherian ring. Moreover, the reference assumed orientability of $\cM$ which is not necessary; see \cite[Theorem A.12]{MW}). Since $t$ is fixed, we may assume without loss of generality that $\cU^{\bullet}$ is bounded and that $\Gamma^t(\cU^{\bullet})_n = \rH^t(\uconf_n(\cM), \bk)$. By Corollary~\ref{cor:derived-saturation}, $\rR\bS(\cU^{\bullet})$ is quasi-isomorphic to a bounded complex $I^{\bullet}$ of induced $\FI$-modules. Since $\bD$ is coherent and $\Gamma^x(I^y)$ is a finitely presented $\bD$-module (Proposition~\ref{prop:induced-connection}), we conclude that $\Gamma^t(\rR\bS(\cU^{\bullet}))$ is finitely presented. The result now follows because the cone of the map $\cU^{\bullet} \to \rR\bS(\cU^{\bullet})$ is quasi-isomorphic to a bounded complex of finitely generated torsion modules (Theorem~\ref{thm:semires} and Corollary~\ref{cor:derived-saturation}) and hence is supported in finitely many degrees. 
\end{proof}

\begin{remark}
	The action of the divided power algebra on $\oplus_{n \ge 0} \rH^t(\uconf_n(\cM), \bk)$ is given by the transfer map: let $f_i \colon \cM^{n+1} \to \cM^n$ be the map that forgets the $i$-th coordinate. Then the multiplication by $x^{[1]}$ map $\rH^t(\uconf_n(\cM), \bk) \to \rH^t(\uconf_{n+1}(\cM), \bk)$ is given by $\sum_{1 \le i \le n+1} f_i^{\ast}$. By naturality of our argument, $M$ (as in the theorem above) can be taken so that the action of the divided power algebra on $M$ agrees (in high enough degree) with the one given by the transfer map.
\end{remark}

\begin{example} 
\label{example:sphere}	Recall that the braid group $B_n$ on $n$ strands is given by the presentation \[ B_n =  \langle \sigma_1, \ldots, \sigma_{n-1} \; \vert \;   \sigma_i \sigma_j = \sigma_j \sigma_i \text{ if } |i - j| >1, \;   \sigma_j \sigma_i \sigma_j = \sigma_i \sigma_j \sigma_i \text{ if } |i-j| = 1 \rangle. \]  It is known that the fundamental group of $\uconf_n(\cS^2)$ is the group obtained from $B_n$ by imposing the additional relation $\sigma_1 \sigma_2\cdots\sigma_{n-1} \sigma_{n-1} \cdots \sigma_2\sigma_1 =1$ (see \cite{sphere}). This relation comes from the fact that a point looping around all the other $n-1$ points in $\uconf_n(\cS^2)$ is contractible (see the figure below).
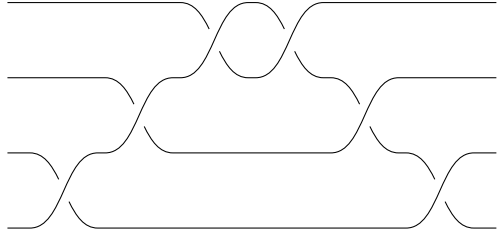
\begin{figure}[h]
\centering
%\captionsetup{labelformat=empty}
\begin{tikzpicture}
\braid[rotate=90, number of strands=4] (braid) a_1 a_2 a_3 a_3 a_2 a_1; 
\end{tikzpicture} 
\caption{Contractible loop in $\uconf_4(\cS^2)$}
\end{figure} 

Since $\uconf_n(\cS^2)$ is path connected, we have for $n \ge 2$, \[\rH_1(\uconf_n(\cS^2), \bZ) \cong \pi_1(\uconf_n(\cS^2))^{\mathrm{ab}}  \cong \bZ/(2n-2).\] The transfer map $\rH_1(\uconf_n(\cS^2), \bZ) \to \rH_1(\uconf_{n-1}(\cS^2), \bZ)$ is given on the class $[\sigma_1]$ in the abelianization by $[\sigma_1] \mapsto (n-2) [\sigma_1]$  (we get 0 if we remove the first or the second strand, and we get $[\sigma_1]$ if we remove any other strand). By the universal coefficient theorem (also see \cite[\S 2.2]{napolitano}),  $\rH^2(\uconf_n(\cS^2), \bZ)$ is naturally isomorphic to \[\Hom_{\bZ}(\rH_1(\uconf_{n}(\cS^2), \bZ), \bQ/\bZ) \cong \bZ/(2n-2),\] and hence the transfer map $\rH^2(\uconf_{n-1}(\cS^2), \bZ) \to \rH^2(\uconf_n(\cS^2), \bZ)$ is given by the following diagram 
	\begin{displaymath}
	\begin{tikzcd}[column sep=1in]
	\rH^2(\uconf_{n-1}(\cS^2), \bZ) \ar{r}{transfer}  \ar{d}{\cong} & \rH^2(\uconf_{n}(\cS^2), \bZ) \ar{d}{\cong} \\
\bZ/(2n-4) \ar{r}{[1] \mapsto [n-1]}  & \bZ/(2n-2) \\
	\end{tikzcd}
	%\caption{} \label{Fig:11} \ar[bend right=15, rightarrow, swap]{rru}{\cor \circ \; \zeta^{\star}}
	\end{displaymath} It follows that $M$, as in Theorem~\ref{thm:configuration-spaces}, can be taken to be the cokernel of the map $\bD[2]  \to \bD[1]$ given by $x^{[0]} \mapsto 2x^{[1]}$. 
\end{example}

\end{document}